\documentclass[a4paper,11pt]{article}
\title{Integral Points Close to Smooth Plane Curves}
\author{ZiAn Zhao (Cody)}
\date{\today} 
\usepackage{graphicx}
\usepackage{float}
\usepackage{framed}
\usepackage{amsmath}
\usepackage{amssymb}
\usepackage{amsfonts}
\usepackage{mathrsfs}
\usepackage{array}
\usepackage{hyperref}
\usepackage{hologo} 
\usepackage{tikz}

\usepackage{graphicx}

\usepackage[numbers]{natbib}
\usepackage{cite}

\usepackage{url}

\usepackage{amsthm}

\theoremstyle{plain} 
\newtheorem{thm}{Theorem}[section] 
\newtheorem{lem}[thm]{Lemma} 
\newtheorem{prop}[thm]{Proposition} 
\newtheorem{cor}[thm]{Corollary}

\theoremstyle{definition}
\newtheorem{defn}[thm]{Definition}

\theoremstyle{remark} 
 
\newtheorem{case}{Case}

\begin{document}
\maketitle 
\begin{abstract}
This is an exposition of a class of problems and results on the number of integral points close to plane curves. We give a detailed proof of a theorem of Huxley and Sargos, following the account of Bordell\`es. Along the way we correct an oversight in the proof, changing some of the explicit values of the constants in the theorem.
\end{abstract}

\tableofcontents 

\pagebreak
\section{Introduction}
The topic of this essay is the methods used to estimate the number of integer points around smooth curves in $\mathbb{R}^2$. In the second chapter, we will look at some relatively weaker statements and eventually build up to a theorem by Huxley and Sargos \citep{huxley1995points}. The explicit statement \citep[Theorem 5.5]{huxley} of the main theorem given by Huxley and Sargos has some issues. In Lemma \ref{7}, the only English-language version \citep[Lemma 5.13]{huxley} suggest that if a point is outside the proper major arc, then the Lagrange polynomial won't intersect it, but this is not the case. If we only fix the erroneous lemma, the proof of the main theorem will have to assume that two major arcs do not intersect for it to work, which is also not clear. However, we have developed some techniques to go around this oversight by proving that two proper majors do not intersect, which fortunately is enough to prove the most important part of the main theorem that the asymptotic inequality still holds. In the third chapter, we will introduce some improvements given by Huxley, Braton, and Tritonov to the main theorem and discuss why they work. In the final chapter, we will introduce two applications of the methods—the square-free number problems and estimating solutions for diophantine inequalities. This essay aims to introduce and understand this class of problems and methods in detail and illustrate them in a way that best fits my understanding. Hopefully, in the continuation of my studies, I will be able to use these interesting methods to solve problems in related subjects. The main source of this essay comes from the book Arithmetic Tales \citep{huxley} by Olivier Bordell\`es, it contains a neatly translated version of results coming from Huxley and Sargos and others in the related topics. Beside changes in proves and statements cause by the oversight we mentioned, most of other proofs are also modified to some degree.

\section{The Theorem of Huxley and Sargos}

One of the goals of this essay is to discuss a certain class of problems in regards of estimating the number of integral points around a smooth curve. At the start, we will take a look at some less generalized methods with strict conditions, their appearances  follow directly from the classical mean value theorem and the divided differences. The problem is then complicated by the appearance of major arcs which allows us to bound the integers in a different way, the goal of the first chapter is to use tools and definitions we will introduce to prove a finer version of the theorem  comes from Huxley and Sargos involving major arcs \citep{huxley1995points}, which is considered to be a more generalized method.
In this chapter, we may assume that $N\geq 4$ is an integer and $\delta$ and $c_0$ are small positive real numbers such that $0<\delta<\frac{1}{4}$. We introduce the following notations
\begin{align*}
\mathcal{S} (f,N,\delta) = \lbrace n\in [N,2N] \cap Z :\parallel f(n) \parallel<\delta) \rbrace.
\end{align*}
$\parallel f(n) \parallel<\delta$ means the nearest integer point has distance less than $\delta$ to $f(n)$ i.e. $(n,y)$ is the nearest integral point to $(n,f(n))$ where $y$ is an integer such that $|f(n)-y|<\delta$. Also
\begin{align*}
\mathcal{R}(f,N,\delta)=|\mathcal{S}(f,N,\delta)|.
\end{align*}
Notice the trivial bound $\mathcal{R}(f,N,\delta)\leq N+1$.

The theorem of Huxley and Sargos will be stated here as a motivation, the proof of the theorem will come in at the end of this chapter when we are in a position to do so. The theorem below
is stated differently from the book Arithmetic Tales \citep[p. 373--374 Theorem 5.5]{huxley} and the original paper by Huxley and Sargos \citep{huxley1995points} and the reason why is discussed in the introduction. 

\begin{thm} [Theorem of Huxley and Sargos]\label{thm1}
Let $k\geq3$ be an integer and $f\in C^k[N,2N]$ such that there exist $\lambda_k>0$ and $c_k\geq 1$ such that, for all $x\in [N,2N]$, we have
\begin{align*}
\lambda_k\leq |f^{(k)}(x)|\leq c_k\lambda_k. \label{1}\tag{1}
\end{align*}
Let $0<\delta<\frac{1}{4}$ be a real number. Then

\begin{itemize}
\item If the set of points in each proper major arc with denominator $\leq (a_k\delta)^{-1}$ can be covered with interval of length $L_j(a_kq_j\delta)^{-1/k}$ which is disjoint from all such covers for other major arcs, where $L_j$ and $q_j$ are the length and denominator associated with their proper major arcs:
\begin{align*}
\mathcal{R}(f,N,\delta)\leq \alpha_k N\lambda_k^{\frac{2}{k(k+1)}}+\beta_kN\delta^{\frac{2}{k(k-1)}}+8k^3(\frac{\delta}{\lambda_k}\Big)^{\frac{1}{k}}+2k^2(5e^3+1);
\end{align*}
where
\begin{align*}
\alpha_k=2k^2c_k^{\frac{2}{k(k+1)}} \quad\text{and} \quad \beta_k=4k^2(5e^3c_k^{\frac{2}{k(k-1)}}+1);
\end{align*}
\item otherwise:
\begin{align*}
\mathcal{R}(f,N,\delta)\leq &  \alpha_k N\lambda_k^{\frac{2}{k(k+1)}}+\beta_k N\delta^{\frac{2}{k(k-1)}}+16k^3(\frac{\delta}{\lambda_k}\Big)^{\frac{1}{k}}+2k^2(5e^3+1);
\end{align*}
where
\begin{align*}
\alpha_k=2k^2c_k^{\frac{2}{k(k+1)}} \quad\text{and} \quad \beta_k=30e^3k^2c_k^{\frac{2}{k(k-1)}}+4k^2;
\end{align*}

\end{itemize}
or alternatively we could write and use the above results as
\begin{align*}
\mathcal{R}(f,N,\delta)\ll N\lambda_k^{\frac{2}{k(k+1)}}+N\delta^{\frac{2}{k(k-1)}}+\big( \frac{\delta}{\lambda_k}\big)^\frac{1}{k}+1.
\end{align*}
\end{thm}
Using Vinogradov notation, $f(x)\ll g(x)$ is equivalent to $f(x)=O(g(x)$, also the implied constant above depends only on $k$ and $c_k$.
We will have an in depth look on how we can use this method, if there exists any weaker version of it and could we improve it. At the end of this section, we will try and prove Theorem \ref{thm1}.
\subsection{Lagrange Interpolation Polynomial} 
First, we will introduce an interpolation method. That is, for a finite set of data points, we will try to construct a function which must go exactly through the values at these data points. The reason why this method will be useful is, for the following sections, we will try to bound the number of certain integer points around a smooth curve. To avoid a trivial bound we could always assume there are more than $k$ such points. By taking multiple intervals which each contain $k$ points, we could reduce the question to bounding the length of the interval by using the classical mean value theorem or a generalization of the mean value theorem which involves using Lagrange polynomials. This would be a common trick to give a bound through out the essay.\\
\linebreak
Here is a simple construction of such a polynomial. Let's say it interpolates at the points $x_0,x_1,x_2,x_3,\dots,x_k$ of some function $f$. Consider a sequence of functions such that each function equals 0 at all the points except at one point it equals $f(x_i)$. By adding these functions together, we get a polynomial passing through exactly these $k+1$ points. Following some straightforward computations we get
\begin{align*}
f_i(x) = \begin{cases} 1, & x=x_i \\ 0, & x=x_j, i\neq j\end{cases}&\Rightarrow
f_i(x)=\prod_{j=0,i\neq j}^{k}\frac{x-x_j}{x_i-x_j}f(x_i)\\
&\Rightarrow \mathcal{P}(x)=\sum_{i=0}^
{k}  \Big( \prod_{j=0,i\neq j}^{k}\frac{x-x_j}{x_i-x_j}f(x_i) \Big). \label{2}\tag{2}
\end{align*}
The leading coefficient of $\mathcal{P}(x)$, $b_k$, is called the divided difference of $f$ at the points $x_0,x_1,x_2,x_3,\dots,x_k$ and it is denoted by $f[x_0,x_1,\dots,x_k]$. With some brute force calculation, we can show that
\begin{align*} 
b_k=\sum_{j=0}^{k} \frac{f(x_j)}{\prod_{0\leq i\leq k,i\neq j}(x_j-x_i)}=\frac{A}{\prod_{0\leq i<j\leq k}(x_j-x_i)} \tag{3},\label{3}
\end{align*}
with
\begin{align*} \label{4}\tag{4}
A=\begin{vmatrix}
1 & 1 & \cdots & 1\\
x_0 & x_1 & \cdots &x_k\\
\vdots & \vdots & \ddots & \vdots\\
x_0^{k-1} & x_1^{k-1} & \cdots & x_k^{k-1}\\
f(x_0) & f(x_1) & \cdots & f(x_k)
\end{vmatrix}
\begin{matrix} \vphantom{x}\\ \vphantom{y} \\\vphantom{z}\\\vphantom{s}\\ \vphantom{f}.\end{matrix}
\end{align*}

It might seem like the degree of the Lagrange interpolating polynomial is directly related to the number of data points it is fitting, specifically, if you have $n$ data points, the Lagrange interpolating polynomial should be of degree $n-1$. But this is simply not the case, we will give a somewhat trivial example here which both showcase the power of this new tool and to prove the point we are making here. Let $f(x)=x$ and let the points we will try to interpolate be $x_i=i$ for $i\in \lbrace1,2,3,4 \rbrace$. Following the way we construct the polynomial (\ref{2}), we get 
\begin{align*}
&\frac{(x-2)(x-3)(x-4)}{-6}+\frac{(x-1)(x-3)(x-4)}{-2}2+\frac{(x-1)(x-2)(x-4)}{-2}3\\&+\frac{(x-1)(x-2)(x-3)}{6}4=x.
\end{align*}
To amend our statement, we could say if you have $n$ data points, the Lagrange interpolating polynomial should be of degree $\leq n-1$. The next thing we will do is to justify that Lagrange interpolation polynomial is well defined by proving its existence and uniqueness \citep[p. 271--287]{lagrange1882oeuvres}. 
\begin{defn}
$\mathbb{P}_n[x]$ is a set of all polynomials of degree $n$ and smaller.
\end{defn}
\begin{thm}[Existence and uniqueness]\label{thm9}
Given $n+1$ distinct points $(x_i)_{i=0}^{n}\in [a,b]$ and $n+1$ real numbers $(f_i)_{i=0}^{n}$, there is exactly one polynomial $p\in \mathbb{P}_n[x]$, namely that given by (\ref{2}), such that $p(x_i)=f_i$ for all $i$.
\end{thm}
\begin{proof}
First we define the Lagrange cardinal polynomials for the points $x_0,x_1,\dots,x_n$ as follows
\begin{align*}
l_k(x)=\prod_{i=0,i\neq k}^{n}\frac{x-x_i}{x_k-x_i},\quad k=0,1,\dots,n.
\end{align*}
Each $l_k$ is the product of $n$ linear factors, hence $l_k\in \mathbb{P}_n[x]$, and from (\ref{2}), $p\in \mathbb{P}_n[x]$. We have shown by construction before such $l_k(x_k)=1$ and $l_k(x_j)=0$ for $j\neq k$. Hence
\begin{align*}
p(x_j)=\sum_{k=0}^{n} f_kl_k(x_j)=f_j, \quad j=0,\dots,n,
\end{align*}
thus $p$ is a polynomial interpolate at our given data points.
For the uniqueness part, let's suppose both $p\in \mathbb{P}_n[x]$ and $q\in \mathbb{P}_n[x]$ interpolate same $n+1$ data points. The polynomial $r=p-q$ is of degree at most n and equals to 0 at $n+1$ distinct points. Such a polynomial can only be the zero polynomial, therefore $r\equiv 0$ and thus the interpolating polynomial is unique.
\end{proof}
Notice that the result for uniqueness is only in one direction, for $n+1$ distinct points, you will get an unique Lagrange polynomial associated to it. But for a given Lagrange polynomial, it could interpolate a different set of points at the same time. The example we discussed before is a good illustration of this. \\
\linebreak
The Lagrange interpolation polynomials are often the appropriate forms to use when we wish to manipulate the interpolation polynomial as part of a larger mathematical expression. However, they are not ideal for numerical evaluation, both because of the speed of calculation and because of the accumulation of rounding error. But it is good enough to give us a practical bound in the following theorems.

\subsection{The First, Second and $k$th Derivative Test} 
The initial aim of theorem \ref{thm1} was to improve a restriction on the $k$th derivative test, which we will state more clearly later. The $k$th derivative test was intended to generalize the first and the second derivative test. We will go through them briefly. First we will state and prove the first derivative test as follows \citep[p. 364--366 Theorem 5.3]{huxley}.
\begin{defn}
$[x]$ is the nearest integer below $x$ such that $x-1<[x]\leq x$.
\end{defn}
\begin{thm}[First derivative test]\label{thm2}
Let $f\in C^1[N,2N]$ such that there exists $\lambda_1>0$ and $c_1\geq 1$ such that for all $x\in [N,2N]$, we have
\begin{align*}
\lambda_1\leq |f'(x)|\leq c_1\lambda_1, \label{5}\tag{5}
\end{align*}
then
\begin{align*}
\mathcal{R}(f,N,\delta)\leq 2c_1N\lambda_1+4c_1N\delta+\frac{2\delta}{\lambda_1}+1.
\end{align*}
In practice, using Titchmarsh-Vinogradov notation, the result can be stated as follows. If
\begin{align*}
|f'(x)|\asymp \lambda_1,
\end{align*}
then we have
\begin{align*}
\mathcal{R}(f,N,\delta)\ll N\lambda_1+N\delta+\frac{\delta}{\lambda}+1.
\end{align*}
$f\asymp g$ is equivalent to $f\ll g$ and $g\ll f$.
\end{thm}
The result above follows naturally from the classical mean-value theorem. It's rather restrictive and becomes useless when $\lambda_1$ is big, but it is a starting point of more elaborate estimates. 

Also the statement itself could be a bit more general, since the proof only invokes the property of $f$ being continuous and differentiable, the derivative $f'$ is not necessarily needed to be continuous.

\begin{proof}
First we would like to get rid of some trivial cases. If $4c_1\delta\geq 1$, then $4c_1\delta+1\geq N+1\geq \mathcal{R}(f,N,\delta)$, similar thing can be said for $2c_1\lambda_1\geq 1$. So we may as well suppose that max$(4c_1\delta,2c_1\lambda_1)<1$. 
We can also assume that there is more than one point in the set otherwise the inequality is trivial. Let $n$ and $n+a$ be any integers in $\mathcal{S}(f,N,\delta)$ (not necessarily consecutive), using the mean value theorem we could show that either
\begin{align*}
a>\frac{1}{2c_1\lambda_1}=a_1 \label{6}\tag{6};
\end{align*}
or
\begin{align*}
a<\frac{2\delta}{\lambda_1}=a_2 \label{7}\tag{7};
\end{align*}
we assume the above is true for now and continue, will come back to that later. Also notice that max($4c_1\delta$,$2c_1\lambda_1$)$<1$ implies 
\begin{align*}
a_1=\frac{1}{2c_1\lambda_1}>max(\frac{2\delta}{\lambda_1},1)
=max(a_2,1). 
\end{align*}
More specifically, $a_1>1$, therefore we can sub-divide the interval $[N,2N]$ into $s=[\frac{N}{a_1}]+1$ sub-intervals $\mathcal{I}_1,\dots ,\mathcal{I}_s$ where each interval has length $\leq a_1$. Which means if there is more than one element from $\mathcal{S}(f,N,\delta)$ inside the interval $\mathcal{I}_j$, we would have a distance between any of the two elements $\leq a_1$ (not necessarily $2$ consecutive elements) and by (\ref{6}) and (\ref{7}), we have their distance $\leq a_2$. For any two elements inside the interval, choose the pair with the largest distance, we would still have the previous result. Let $n$ and $n+a$ be such a pair with $a\leq a_2$. Then if there are more elements lie in the interval, they must lie in between $n$ and $n+a$. Thus, we could infer the following to be true,
\begin{align*}
|\mathcal{S}(f,N,\delta)\cap\mathcal{I}_j|\leq a_2+1,
\end{align*}
which still holds if there is only one element inside the interval. From here we could easily deduce the result.
\begin{align*}
\mathcal{R}(f,N,\delta)\leq\Big(\frac{N}{a_1}+1\Big)(a_2+1)=2c_1N\lambda_1+4c_1N\delta+\frac{2\delta}{\lambda_1}+1.
\end{align*}
Now we will try to end the argument by deducing (\ref{7}) and (\ref{6}) to be true by using the mean value theorem. For any two elements $n$ and $n+a$ in $\mathcal{S}(f,N,\delta)$, there exist $m_1$ and $m_2$ integers and $\delta_1$ and $\delta_2$ real numbers such that $f(n)=m_1+\delta_1$ and $f(n+a)=m_2+\delta_2$ with $|\delta_i|<\delta$ for $i \in \lbrace1,2\rbrace$. Which means there exists $m\in \mathbb{Z}$ and $\delta_3\in \mathbb{R}$ such that $|\delta_3|<2\delta$ and
\begin{align*}
f(n+a)-f(n)=m+\delta_3.
\end{align*}
Using the mean value theorem on $f$, we get for $t\in[N,2N]$ the following
\begin{align*}
f(n+a)-f(n)=af'(t) \quad\text{and}\quad af'(t)=m+\delta_3.
\end{align*}
If $m\neq 0$, since $m\in \mathbb{Z}$, we have $|m|\geq 1$. From (\ref{5}) we could deduce
\begin{align*}
ac_1\lambda_1\geq a|f'(t)|=|m+\delta_3|\geq |m|-|\delta_3|>1-\frac{1}{2}=\frac{1}{2},
\end{align*} 
which gives us (\ref{6}). If $m=0$, by (\ref{5}) we have
\begin{align*}
a\lambda_1\leq a|f'(t)|=|\delta_3|<2\delta,
\end{align*}
which gives (\ref{7}).
\end{proof}
For the next result, we will assume $0<\delta<\frac{1}{8}$. First we will introduce the following lemma. \citep[p. 383--385 Lemma 5.7]{huxley}
\begin{lem}[Reduction principle]\label{lem1} Let $f:[N,2N]\longrightarrow \mathbb{R}$ be any map, $A$ be a real number satisfying $1\leq A\leq N$ and, for all integers $a\in [1,A]$, we define on $[N,2N-a]$ the function $\Delta_af$ by
\begin{align*}
\Delta_af(x)=f(x+a)-f(x).
\end{align*}
Then
\begin{align*}
\mathcal{R}(f,N,\delta)\leq \frac{N}{A}+\sum_{a\leq A}\mathcal{R}(\Delta_af,N,2\delta)+1.
\end{align*}
\end{lem}
\begin{proof}
For all $a\in\mathbb{N}$, define
\begin{align*}
S(a)=\lbrace n\in [N,2N]\cap \mathbb{Z}:n \text{ and }n+a \text{ are consecutive in }\mathcal{S}(f,N,\delta)\rbrace.
\end{align*}
As before, assume more than one point in the set $\mathcal{S}(f,N,\delta)$, otherwise the statement is trivial. We observe that for any integer inside  $\mathcal{S}(f,N,\delta)$, except the largest one, has a successive element and lies in only one subset $\mathcal{S}(a)$, so we could easily deduce the following,
\begin{align*}
\mathcal{R}(f,N,\delta)=\sum_{a=1}^{\infty}|\mathcal{S}|+1=\sum_{a\leq A}|\mathcal{S}(a)|+\sum_{a> A}|\mathcal{S}(a)|+1. \tag{8}\label{8}
\end{align*}
We write the set as $\mathcal{S}(f,N,\delta)=\lbrace n_1\leq n_2\leq \cdots\leq n_k\rbrace$ and define $d_j$ for $j\in \lbrace 1,\dots,k-1 \rbrace$ as
\begin{align*}
d_1=n_2-n_1 \text{, } d_2=n_3-n_2 \text{,\dots,}d_{k-1}=n_k-n_{k-1}.
\end{align*}
By definition we have $|\mathcal{S}(a)|$ counts the number of elements with distance a, that is the number of indexes $j\in \lbrace 1,\dots,k-1\rbrace$ such that $d_j=a$, so clearly
\begin{align*}
\sum_{a=1}^{\infty}a|\mathcal{S}(a)|=\sum_{j=1}^{k-1}d_j=\sum_{j=1}^{k-1}(n_{j+1}-n_j)=n_k-n_1\leq N.
\end{align*}
Therefore we can trivially bound the second term of (\ref{8}) by the following,
\begin{align*}
N\geq \sum_{a=1}^{\infty}a|\mathcal{S}(a)|\geq \sum_{a>A}a|\mathcal{S}(a)|\geq A\sum_{a>A}|\mathcal{S}(a)| \Rightarrow \sum_{a>A}|\mathcal{S}(a)|\leq \frac{N}{A}.
\end{align*}
Let $n\in \mathcal{S}(a)$, which means $n$ and $n+a$ are consecutive in $\mathcal{S}(f,N,\delta)$, so that
\begin{align*}
||\Delta_af(n)||= ||f(n+a)-f(n)||\leq ||f(n+a)||+||f(n)||<2\delta.
\end{align*}
To justify the first inequality above to be true, simply take\\ $|| (f(n+a) \text{ mod } 1)-(f(n) \text{ mod }1)||$, it gives the same value as $||f(n+a)-f(n)||$. So we have $n\in \mathcal{S}(\Delta_af,N,2\delta)$, that is
\begin{align*}
|\mathcal{S}(a)|\leq \mathcal{R}(\Delta_af,N,2\delta),
\end{align*}
which gives a bound for first term of (\ref{8}) and the result.
\end{proof}
For the second derivative test, the idea is to use the mean value theorem to get $| (\Delta_af)'(x)|\asymp a\lambda_2$, use the first derivative test on $\Delta_af$ and use the previous lemma to get back to $f$. This gives the following theorem \citep[p. 385--386 Theorem 5.4]{huxley}.
\begin{thm}[Second derivative test]\label{thm3} 
Let $f\in C^2[N,2N]$ such that there exist $\lambda_2>0$ and $c_2\geq 1$ such that, for all $x\in [N,2N]$, we have
\begin{align*}
\lambda_2\leq|f''(x)|\leq c_2\lambda_2, \tag{9}\label{9}
\end{align*}
and 
\begin{align*}
N\lambda_2\geq c_2^{-1} \tag{10}. \label{10}
\end{align*}
Then
\begin{align*}
\mathcal{R}(f,N,\delta)\leq 6\lbrace (3c_2)^\frac{1}{3}N\lambda_2^\frac{1}{3}+(12c_2)^\frac{1}{2}N\delta^\frac{1}{2}+1\rbrace, \label{11} \tag{11}
\end{align*}
or to put it in a more practical way, if
\begin{align*}
|f''(x)|\asymp \lambda_2 \text{ and } N\lambda_2 \gg 1,
\end{align*}
then
\begin{align*}
\mathcal{R}(f,N,\delta)\ll N\lambda_2^\frac{1}{3}+N\delta^\frac{1}{2}+1.
\end{align*}
\end{thm}
\begin{proof}
As per usual, we first get rid of the trivial cases. If $\lambda_2\geq (3c_2)^{-1}$ or $\delta\geq (12c_2)^{-1}$, then it would result in the right hand side of the inequality (\ref{10}) being $\geq N+1$, which means the statement would be trivially true. So we could assume that
\begin{align*}
0<\lambda_2<(3c_2)^{-1} \text{ and } 0<\delta<(12c_2)^{-1}. \tag{12} \label{12}
\end{align*}
Let $A\in\mathbb{R}$ such that $1\leq A\leq N$. For all $x\in [N,2N]$ and all $a\in[1,A]\cap\mathbb{Z}$ such that $x+a\in [N,2N]$, applying the mean value theorem on them, that
\begin{align*}
(\Delta_af)'(x)=f'(x+a)-f'(x)=af''(t)
\end{align*}
for some $t\in[x,x+a]\subseteq [N,2N]$, with the condition from (\ref{9}), we would have for all $x\in [N,2N]$ and all $a\in [1,A]\cap \mathbb{Z}$ such that $x+a\in[N,2N]$, we could easily find
\begin{align*}
a\lambda_2\leq |(\Delta_af)'(x)|\leq c_2a\lambda_2.
\end{align*}
By the first derivative test we get
\begin{align*}
\mathcal{R}(\Delta_af,N,2\delta)\leq 2c_2Na\lambda_2+8c_2N\delta+\frac{4\delta}{a\lambda_2}+1.\label{13}\tag{13}
\end{align*}
Even though $\Delta_af$ might not be defined on some points in $[N,2N]$, the inequality still holds because the range where $\Delta_af$ is defined is contained in $[N,2N]$. We could simply extend the undefined part to be any smooth function. As we are discussing an upper bound in (\ref{12}), it still works, since by the nature of the argument, we only care about counting the integer points on the defined interval. Now we combine the above with Lemma \ref{lem1}. We get the following,
\begin{align*}
\mathcal{R}(f,N,\delta)\leq \frac{N}{A}+\sum_{a\leq A}\Big( 2c_2Na\lambda_2+8c_2N\delta+\frac{4\delta}{a\lambda_2}+1\Big)+1. \tag{14}\label{14}
\end{align*}
By condition (\ref{10}) and the fact that $a$ is a positive integer, we get
\begin{align*}
4c_2N\delta\geq 4\delta\lambda_2^{-1}\geq 4\delta(a\lambda_2)^{-1} \text{ and } 1\leq c_2Na\lambda_2.
\end{align*}
Replace terms on the right side of the inequality (\ref{14}) using the above two inequalities, we get
\begin{align*}
\mathcal{R}(f,N,\delta)&\leq \frac{N}{A}+\sum_{a\leq A}\Big( 3c_2Na\lambda_2+12c_2N\delta\Big)+1\\
&\leq \frac{N}{A}+3c_2NA^2\lambda_2+12c_2NA\delta+1. \tag{15}\label{15}
\end{align*}
To improve our bound even further, we need to try to optimize the choice of $A$ and see if there exists $A\in[1,N]$ such that the right hand side of the above inequality could be as small as it can. The following lemma \citep{srinivasan1962van} does exactly the thing we want, optimizing the choice of $H$ in a certain interval. An explicit statement of the lemma can also be found in \citep[p. 363 Lemma 5.6]{huxley}.
\begin{lem}[Srinivasan]\label{lem2}
Let
\begin{align*}
E(H)=\sum_{i=1}^{m}A_iH^{a_i}+\sum_{j=1}^{n}B_jH^{-bj},
\end{align*}
where m,n $\in$ $\mathbb{N}$ and $A_i$, $B_j$, $a_i$ and $b_j$ are positive real numbers. Suppose that $0\leq H_1\leq H_2$. Then
\begin{align*}
\min_{H_1\leq H\leq H_2}E(H)\leq (m+n)\Big\lbrace \sum_{i=1}^{m}\sum_{j=1}^{n}(A_i^{b_j}B_j^{a_i})^{\frac{1}{a_i+b_j}}+\sum_{i=1}^{m}A_iH_1^{a_i}+\sum_{j=1}^{n}B_jH_2^{-b_j}\Big\rbrace.
\end{align*}
\end{lem}
Now we let $n=1,m=2,b_1=1,a_1=1,a_2=2,B_1=N,A_1=12c_2N\delta,A_2=3c_2N\lambda_2,H=A$
\begin{align*}
\min_{1\leq A\leq N} \Big(& NA^{-1}+3c_2N\lambda_2A^2+12c_2NA\delta\Big) \\
&\leq 3\Big\lbrace (A_1N^{a_1})^{\frac{1}{a_1+1}}+(A_2N^{a_2})^{\frac{1}{a_2+1}}  +A_1H_1^{a_1}+A_2H_1^{a_2}+B_1H_2^{-b_1}\Big\rbrace\\
&=3N\Big\lbrace (3c_2\lambda_2)^{\frac{1}{3}} +  3c_2\lambda_2  +(12c_2\delta)^{\frac{1}{2}}+12c_2\delta \Big\rbrace +3. \tag{16}\label{16}
\end{align*}
Now (\ref{12}) implies that
\begin{align*}
0<3c_2\lambda_2<1 \text{ and } 0<12c_2\delta<1,
\end{align*}
which means $(3c_2\lambda_2)^\frac{1}{3}>3c_2\lambda_2$ and $(12c_2\delta)^\frac{1}{2}>12c_2\delta$, with all that said, bound in (\ref{15}) can be optimized as the following,
\begin{align*}
\mathcal{R}(f,N,\delta) &\leq 3N\Big\lbrace (3c_2\lambda_2)^{\frac{1}{3}} +  3c_2\lambda_2  +(12c_2\delta)^{\frac{1}{2}}+12c_2\delta \Big\rbrace +4\\
&\leq 6\Big\lbrace(3c_2\lambda_2)^\frac{1}{3}N+(12c_2\delta)^\frac{1}{2}N+1\Big\rbrace,
\end{align*}
which is the result.
\end{proof}
The natural question to ask now is if the previous two results could be generalized in some way. In order to do so, we should first consider a more generalized version of the mean value theorem which put higher derivatives into play. It is stated as follows \citep[p.380--381 Theorem 5.2]{huxley}.
\begin{thm}[Divided differences]\label{thm4}
Let $k$ be a positive integer, $x_0<x_1<\cdots<x_k$ be real numbers and $f\in C^k[x_0,x_k]$. Set $\mathcal{P}(x)=b_kx^k+\cdots+b_0$ the unique polynomial of degree $\leq k$ such that $\mathcal{P}(x_i)=f(x_i)$ for $i=0,\dots,k$. Then there exists a real number $t\in [x_0,x_k]$ such that
\begin{align*}
b_k=\frac{f^{(k)}(t)}{k!}.
\end{align*}
\end{thm}
The proof of the theorem is rather straightforward. By letting $F(x)=f(x)-\mathcal{P}(x)$, we can use induction to prove that for any $F(x)\in C^k[x_0,x_k]$ such that $F(x_0)=F(x_1)=\cdots=F(x_k)$. Then there exists a real number $t\in[x_0,x_k]$ such that $F^{(k)}(t)=0$. Notice that $\mathcal{P}^{(k)}(x)=k!b_k$. \\
\linebreak
The polynomial $\mathcal{P}$ mentioned above is a Lagrange polynomial. We can see from (\ref{4}) that $A\in \mathbb{Z}$ if $x_j\in \mathbb{Z}$ and $f(x_j)\in \mathbb{Z}$. When $k=1$, by (\ref{3}), we can find $b_k=\frac{f(x_1)-f(x_0)}{x_1-x_0}$, which shows that Theorem \ref{thm4} generalizes the mean value theorem. With that in mind, now we can go back to the generalization of the first and second derivative tests. The $k$th derivative test \citep[p. 387--372 Proposition 5.1]{huxley} is stated as follows
\begin{thm}[$k$th derivative test]\label{thm5}
Let $k\geq 1$ be an integer and $f\in C^k[N,2N]$ such that there exist $\lambda_k>0$ and $c_k\geq 1$ such that, for all $x\in[N,2N]$, we have
\begin{align*}
\lambda_k\leq |f^{(k)}(x)|\leq c_k\lambda_k. \label{17}\tag{17}
\end{align*}
Assume also that
\begin{align*}
(k+1)!\delta<\lambda_k. \label{18}\tag{18}
\end{align*}
If $\alpha_k=2k(2c_k)^{\frac{2}{k(k+1)}}$, then
\begin{align*}
\mathcal{R}(f,N,\delta)\leq \alpha_kN\lambda_k^{\frac{2}{k(k+1)}}+4k.
\end{align*}
To state the theorem in a more practical sense, we can say if
\begin{align*}
|f^{(k)}(x)|\asymp \lambda_k \text{ and }\delta\ll \lambda_k,
\end{align*}
then
\begin{align*}
\mathcal{R}(f,N,\delta)\ll N\lambda_k^{\frac{2}{k(k+1)}}+1.
\end{align*}
\end{thm}
\begin{proof}
As usual, we first get rid of the trivial case, if $\lambda_k\geq \frac{1}{2}$, then
\begin{align*}
\alpha_kN\lambda_k^{\frac{2}{k(k+1)}}+1\geq N+1 \geq \mathcal{R}(f,N,\delta).
\end{align*}
Since $k\geq 1$ and $c_k\geq 1$. So we can assume that $\lambda_k<\frac{1}{2}$ and that the size of $\mathcal{S}(f,N,\delta)$ is more than $4k$. The way we try to generalize the first derivative test here is by taking $k+1$ consecutive points $n<n+a_1<n+a_2<\cdots<n+a_k$ in $\mathcal{S}(f,N,\delta)$ and say the distance between the first and the last elements $a_k$ satisfies 
\begin{align*}
a_k\geq 2k\alpha_k^{-1}\lambda_k^{-\frac{2}{k(k+1)}}. \label{19}\tag{19}
\end{align*}
As we did in the proof of the first derivative test, we will assume (\ref{19}) to be true for now and will come back and prove it later. Let's write the set in the form $\mathcal{S}(f,N,\delta)=\lbrace n<n+a_1<\cdots<n+a_k<\cdots<n+a_z\rbrace$ where $z > 4k$. Let $T$ be a subset of $\mathcal{S}(f,N,\delta)$ such that it contains each $(k+1)$th elements of $\mathcal{S}(f,N,\delta)$, that is 
\begin{align*}
T=\lbrace n+a_{mk+(m-1)}\in \mathcal{S}(f,N,\delta): m\in\mathbb{N}\setminus\lbrace0\rbrace  \text{ and }mk+(m-1)\leq z  \rbrace. 
\end{align*}
We can see by (\ref{19}), for any two elements in the set $T$, they differ by more than $d_k=2k\alpha_k^{-1}\lambda_k^{-\frac{2}{k(k+1)}}$ and $k\geq 1$, therefore we can deduce the following, which is essentially the result
\begin{align*}
\mathcal{R}(f,N,\delta)\leq (k+1)(|T|+1)\leq 2k\Big(\frac{N}{d_k}+2\Big).
\end{align*}
Now our job is to show that the inequality (\ref{19}) is indeed true and then we are done. We may abuse the notations a little bit and let $n, n+a_1,\dots,n+a_k$ be any $k+1$ consecutive elements in $S(f,N,\delta)$. By definition, there exist integers $m_0,\dots,m_k$ and real numbers $\delta_0,\dots,\delta_k$ such that $|\delta_j|<\delta$ for all $j\in \lbrace0,\dots,k\rbrace$ and $f(n+a_j)=m_j+\delta_j$ where $a_0=0$. Using Theorem \ref{thm4} and (\ref{3}), let $\mathcal{P}(x)=b_kx^k+\cdots+b_0$ be the Lagrange polynomial interpolating the points $(n+a_j,m_j+\delta_j)$. We know there exist $t\in [n, n+a_k]$ such that 
\begin{align*}
b_k=\sum_{j=0}^{k} \frac{m_j+\delta_j}{\prod_{0\leq i\leq k,i\neq j}(a_j-a_i)}=\frac{f^{(k)}(t)}{k!};     \label{20}\tag{20} 
\end{align*}
let $P=b'_kX^k+\cdots+b'_0$ be the polynomial interpolating the points $(n+a_j,m_j)$, then we also have
\begin{align*}
b'_k=\sum_{j=0}^{k} \frac{m_j}{\prod_{0\leq i\leq k,i\neq j}(a_j-a_i)}=\frac{A_k}{D_k};
\end{align*}
where we can see $A_k\in \mathbb{Z}$ and $D_k=\prod_{0\leq i<j\leq k}(a_j-a_i)>0$, and we can get
\begin{align*}
b'_k=b_k-\sum_{j=0}^{k}\frac{\delta_j}{\prod_{0\leq i\leq k,i\neq j}(a_j-a_i)};
\end{align*}
and then 
\begin{align*}
k!A_k=k!D_kb'_k=D_kf^{(k)}(t)-k!D_k\sum_{j=0}^{k}\frac{\delta_j}{\prod_{0\leq i\leq k,i\neq j}(a_j-a_i)}=x+y.
\end{align*}
Since $|\delta_j|<\delta$ and $|a_j-a_i|\geq 1$ for all $i\neq j$, we have
\begin{align*}
|y|<k!D_k\delta\sum_{j=0}^{k}\frac{1}{\prod_{0\leq i\leq k,i\neq j}|a_j-a_i|}\leq (k+1)!D_k\delta.
\end{align*}
Now from (\ref{17}), (\ref{18}) and (\ref{20}), the above becomes
\begin{align*}
|y|< \lambda_kD_k\leq D_k |f^{(k)}(t)|=|x|.
\end{align*}
It is easy to see that $|x|>\frac{1}{2}$ since $x+y=k!A_k$ is an integer, because $|x|>|y|$ which means $x+y\neq 0$ and $1\leq |x+y|\leq |x|+|y|<2|x|$. Also we can see $a_0<a_1<\cdots<a_k$, so 
\begin{align*}
D_k=\prod_{0\leq i<j\leq k}(a_j-a_i)\leq \prod_{0\leq i<j\leq k}a_j\leq \prod_{0\leq i<j\leq k}a_k\leq a_k^{\frac{k(k+1)}{2}}. 
\end{align*}
With (\ref{17}), we get
\begin{align*}
\frac{1}{2}<D_k|f^{(k)}(x)|\leq a_k^{\frac{k(k+1)}{2}}c_k\lambda_k=\frac{\lambda_k}{2}((2k)^{-1}a_k\alpha_k)^{\frac{k(k+1)}{2}}
\end{align*}
which implies (\ref{19}), and concludes the proof.
\end{proof}
The $k$th derivative test stated above does generalize the first derivative test, but we can see that many bounds we used in the proof could potentially be improved, such as $|a_j-a_i|$ could be a larger number than $1$ in many special cases, which would enable us to give a better upper bound for $|y|$, that means the condition (\ref{18}) could be less restrictive. Naturally, the quest to establish a $k$th derivative test without the condition (\ref{18}) arises. This leads us to the result given by Huxley and Sargos which we have stated in Theorem \ref{thm1}. To prove it is rather difficult, we would need several tools first and then we will be in a the position to do so.
\subsection{Preparatory Lemmas and Major Arcs} 
The aim of this section is to discuss the techniques and tools that we will be using in the proof of Theorem \ref{thm1}, so we may assume some of the conditions, such as $k\geq 3$, which are stated in the Theorem. The first tool we will introduce is an easy enumeration principle \citep[p. 374--375 Lemma 5.9]{huxley}.
\begin{lem}\label{lem3}
Let $S$ be a finite set of integers, $S=\lbrace a_1<a_2\cdots<a_n\rbrace$ such that $|a_n-a_1|\leq N$. If one can cover $S$ by pairwise distinct intervals $\mathcal{I}$ and if $L(\mathcal{I})$ is the length of $\mathcal{I}$, then
\begin{align*}
|S|\leq N \max_{\mathcal{I}}\Big( \frac{|S\cap \mathcal{I}|}{L(\mathcal{I})}\Big)+2\max_{\mathcal{I}}|S\cap \mathcal{I}|.
\end{align*}
\end{lem}
\begin{proof}
Let $\mathcal{I}_1,\dots,\mathcal{I}_J$ be such a covering of $S$, pairwise distinct here simply means if $h\neq j$, then $\mathcal{I}_h \cap \mathcal{I}_j=\emptyset$ and if $2\leq j\leq J-1$, then $\mathcal{I}_j\subseteq [a_1,a_n]$, and what follows next is very easy to see,
\begin{align*}
|S|\leq \sum_{j=1}^{J}|S \cap \mathcal{I}_j|&=\sum_{j=2}^{J-1}\Big\lbrace L(\mathcal{I}_j)\frac{|S \cap \mathcal{I}_j|}{L(\mathcal{I}_j)} \Big\rbrace +|S \cap \mathcal{I}_1| +|S \cap \mathcal{I}_J|\\
& \leq \max_{1\leq j\leq J}\Big( \frac{|S\cap \mathcal{I}_j|}{L(\mathcal{I}_j)}\Big)\sum_{j=2}^{J-1}L(\mathcal{I}_j) + 2\max_{1\leq j\leq J}|S\cap \mathcal{I}_j|\\
&\leq N\max_{1\leq j\leq J}\Big( \frac{|S\cap \mathcal{I}_j|}{L(\mathcal{I}_j)}\Big)+ 2\max_{1\leq j\leq J}|S\cap \mathcal{I}_j|.
\end{align*}
\end{proof}
The next tool we will introduce is a certain class of inequalities which exist to bound a derivatives of some function if we know the bound of its highest order of derivative and they were the lowest. They are called the Landau-Hadamard-Kolmogorov inequalities. It was first developed by E. Landau \citep{landau} in 1913 with rather restricted conditions, explicit statements of the following three theorems can also be found in \citep[p. 374]{huxley}.
\begin{thm} [E. Landau]\label{thm6}
If $I$ is an interval with length $\geq 2$ and $f\in C^2(I)$ satisfies the conditions $|f(x)|\leq 1$ and $|f''(x)|\leq 1$ on $I$, then $|f'(x)|\leq 2$ and the constant 2 is the best possible.
\end{thm}
This was generalized by Hadamard \citep{ako}, without the explicit restrictions on the bound of $|f|$ and $|f''|$ .
\begin{thm} [Hadamard] \label{thm7}
If $f\in C^2([a,a+L])$ and $a\in \mathbb{R}$, $L>0$, then
\begin{align*}
\sup_{a\leq x\leq a+L} |f'(x)|\leq \frac{2}{L}\sup_{a\leq x\leq a+L}|f(x)|+\frac{L}{2}\sup_{a\leq x \leq a+L}|f''(x)|.
\end{align*}
\end{thm}
As with the development of the derivative tests we have introduced, the next natural step was to generalize it for higher order of derivatives \citep{Lned}.
\begin{thm} [L. Neder] \label{thm8}
If $a\in \mathbb{R}$, $L>0$ and $f\in C^k[a,a+L]$, then, for all $j\in \lbrace 1,\dots,k-1\rbrace$, we have
\begin{align*}
\sup_{a\leq x\leq a+L}|f^{(j)}(x)|\leq \frac{(2k)^{2k}}{L^{j}}\sup_{a\leq x\leq a+L}|f(x)|+L^{k-j}\sup_{a\leq x\leq a+L}|f^{(k)}(x)|. \tag{21}\label{21}
\end{align*}
\end{thm}
Theorem \ref{thm8} stated above is sufficient enough for us to use in the proof of Theorem \ref{thm1}. But what we will use is the following improved version due to Gorny \citep{gorny1939contribution}, the original source should also be in Bordell\`es \citep[Lemma 5.10]{huxley}.
\begin{lem}[Gorny]\label{lem6}
Let $k\geq 2$ be an integer, $a\in\mathbb{R}$, $L>0$ and $f\in C^k[a,a+L]$ such that, for all $x\in [a,a+L]$, we have
\begin{align*}
|f(x)|\leq M_0 \quad \text{and} \quad |f^{(k)}(x)|\leq M_k
\end{align*}
with $M_0<\infty$ and $M_k<\infty$. Then, for all $x\in [a,a+L]$ and $j\in \lbrace 1,\dots,k-1\rbrace$, we have
\begin{align*}
|f^{(j)}(x)|<4(e^2k/j)^jM_0^{1-j/k}\big\lbrace max\big( M_k,k!M_0L^{-k}\big) \big\rbrace^{j/k}.
\end{align*}
In practice, we will use this result in the following form
\begin{align*}
|f^{(j)}(x)|<4e(ek/j)^j\big\lbrace k^{j+1}M_0L^{-j}+e^{j-1}M_0^{1-j/k}M_k^{j/k}\big\rbrace. \label{22}\tag{22}
\end{align*}
\end{lem}

The next technique \citep[Lemma 5.11]{huxley} will be used to calculate the implied constant appeared in Theorem \ref{thm1}.
\begin{lem} \label{lem4}
Let $k\geq 3$ be an integer and $a>e(k-1)$ be a real number. Then 
\begin{align*}
\sum_{j=1}^{k-1}\Big( \frac{a}{j}\Big)^{2j}\leq e^2k\Big(\frac{a}{k} \Big)^{2k-2}.
\end{align*}
\end{lem}
\begin{proof}
By some brute force calculation, we can find that the derivative of $\Big( \frac{e(k-1)}{x}\Big)^{2x}$ equals $0$ only when $x=k+1$, and it is a local maximum. Since $a>e(k-1)$, we now can say that the function $x\mapsto (a/x)^{2x}$ is increasing as soon as $1\leq x\leq k-1$ and so
\begin{align*}
\sum_{j=1}^{k-1}\Big( \frac{a}{j}\Big)^{2j}\leq (k-1)\Big(\frac{a}{k-1} \Big)^{2k-2}
\end{align*}
and further more
\begin{align*}
\exp\Big\lbrace \log \Big( \frac{k-1}{k}\Big(\frac{a}{k-1}\times\frac{k}{a}\Big)^{2k-2}\Big)\Big\rbrace=\exp\Big\lbrace (2k-1)\log \Big( 1+\frac{1}{k-1}\Big)\Big\rbrace. \label{23}\tag{23}
\end{align*}
Notice that if we let $\frac{1}{k-1}-\log(1+\frac{1}{k-1})=b$, we will find $-b=\log \Big( \frac{1+\frac{1}{k-1}}{e^{\frac{1}{k-1}}}\Big)$, and we know $1+x\leq e^x$ to be true by the fact that $1+x$ is tangent to $e^x$ at $x=0$ and $e^x$ is convex. So $b\geq 0$, and to continue from (\ref{23}), we have
\begin{align*}
\exp\Big\lbrace (2k-1)\log \Big( 1+\frac{1}{k-1}\Big)\Big\rbrace \leq \exp\Big( \frac{2k-3}{k-1}\Big)\leq e^2. 
\end{align*}
Rearrange things a bit, we will get the result.
\end{proof} 
For the next technique, we will have to dedicate a few rather long pages of explanation to cover the whole idea, first we will write down it's definition \citep[Definition 5.3]{huxley}.
\begin{defn}\label{def1} (Major arc)
\begin{enumerate}
\item A major arc associated to $f^{(k)}$ is a maximal set $\mathcal{A}=\lbrace n_1,\dots,n_J\rbrace$ of consecutive points of $\mathcal{S}(f,N,\delta)$, where $J\geq k+1$ is an integer, such that, for all $j\in \lbrace 1,\dots,J\rbrace$, we have ($\lfloor x\rceil$ is the nearest integer to $x$)
\begin{align*}
\lfloor f(n_j) \rceil=P(n_j) 
\end{align*}
where $P\in \mathbb{Q}[X]$ is the Lagrange polynomial of degree $<k$ interpolating the points $(n_j,\lfloor f(n_j)\rceil)$. The equation $y=P(x)$ is called the equation of $\mathcal{A}$. We set $\mathcal{C}_P$ the curve with equation $y=P(x)$.
\item Let $q$ be the smallest positive integer such that $P\in \frac{1}{q}\mathbb{Z}[X]$. Then $q$ is called the denominator of $\mathcal{A}$.
\end{enumerate}
\end{defn}
Notice that the maximal set of points $\mathcal{A}$ depends on the choices of the Lagrange polynomial and the major arc may not exist for certain such choices. It might be easier to understand what exactly does the definition mean by simply constructing a major arc. Select $k$ elements from $\mathcal{S}(f,N,\delta)$. Apply the formula for Lagrange interpolation polynomial $P$ on these points. Since the elements we are interpolating are all integral points, we don't have to worry about the condition $P\in \mathbb{Q}[X]$. Let $\mathcal{B}$ be the set such that $\mathcal{B}=\lbrace n\in \mathcal{S}(f,N,\delta): \lfloor f(n)\rceil=P(n) \rbrace$. $\mathcal{A}$ is the set which contains the longest sequence of elements in $\mathcal{B}$ that are consecutive in $\mathcal{S}(f,N,\delta)$. To distinguish whether $\mathcal{A}$ is a major arc or not, we need to see if $|\mathcal{A}|\geq k+1$. 

To continue, we give the following notation
\begin{align*}
\mathcal{C}_\delta=\Big\lbrace (x,y)\in [N,2N]\times \mathbb{R}: |y-f(x)|<\delta\Big\rbrace.
\end{align*}
One reasonable question to ask at this point would be what will it look like if $\mathcal{C}_\delta$ and $\mathcal{C}_P$ are being plotted on the same graph, and more specifically, what will their intersections be. The following graph should give us a better idea.
\begin{figure}[H]
\centering
\includegraphics[scale=0.45]{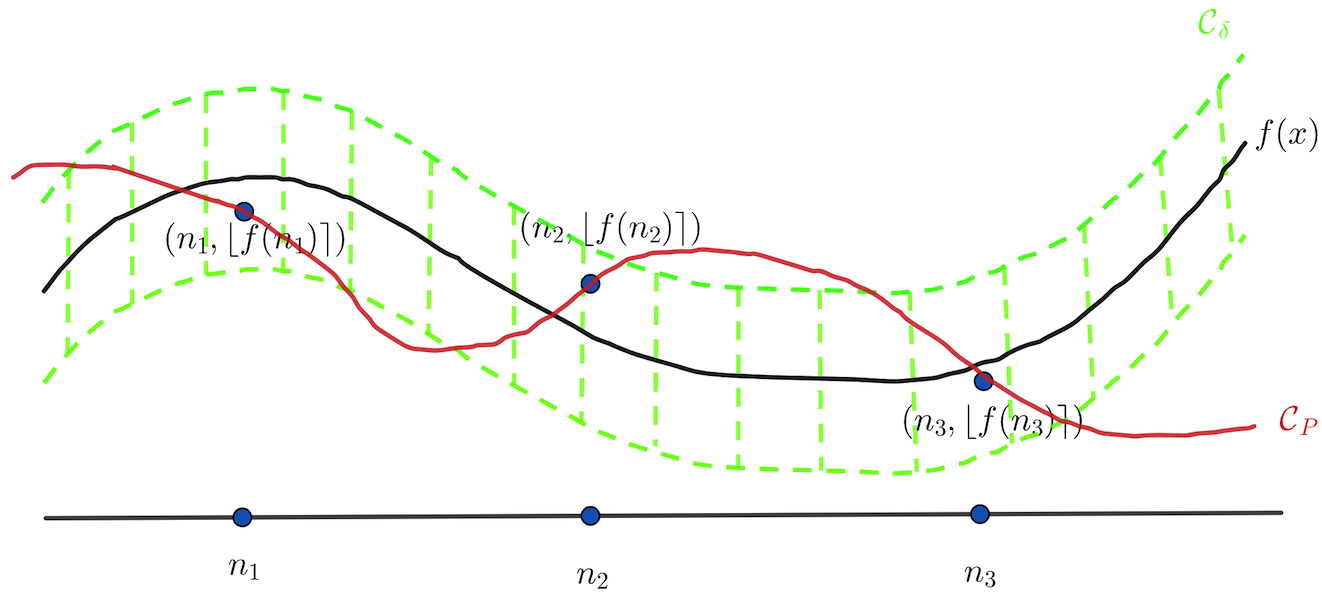}
\caption{$\mathcal{C}_\delta\cap \mathcal{C}_P$}
\label{fig1}
\end{figure}
As figure \ref{fig1} has shown, the region of intersection between $\mathcal{C}_\delta$ and $\mathcal{C}_P$ is being chopped into several connected components. Now, we will state the lemma \citep[Lemma 5.12]{huxley} which gives us a bound for the number of connected components of the set $\mathcal{C}_\delta\cap \mathcal{C}_P$.
\begin{lem}\label{lem5}
The set $\mathcal{C}_\delta\cap \mathcal{C}_P$ has at most $k$ connected components.
\end{lem}
\begin{proof}
First, we expand the region where $f(x)$ is defined. Let $\Tilde{f}\in C^k(\mathbb{R})$ be a function defined as $\tilde{f}(x)=f(x)$ for $x\in [N,2N]$, 
\begin{align*}
\tilde{f}^{(k)}=f^{(k)}(N) \text{ if } x\leq N \text{ and }
\tilde{f}^{(k)}(x)=f^{(k)}(2N) \text{ if } x\geq 2N, 
\end{align*}
so that $|\tilde{f}^{(k)}(x)|\asymp \lambda_k$. Since $\lambda_k>0$, we can see $\tilde{f}(x)$ is just some polynomial of degree $k$ in the interval $(-\infty,N]\cup[2N,\infty)$. Now $P$ is a polynomial of degree $<k$ by definition, so the behaviours of $\tilde{f}$ and $P$ at infinity are different and the set $\tilde{\mathcal{C}_\delta}\cap \mathcal{C}_P$ is bounded. Hence, at both ends of each connected component not reduced to a singleton (a single point intersection), we would have $P(x)=\tilde{f}(x)\pm \delta$ (the set $\tilde{\mathcal{C}_\delta}$ is the analogue of $\mathcal{C}_\delta$ for the function $\tilde{f}$). \\
\linebreak
What follows next is basically a consequence of divided differences, assume there are $k+1$ connected components not reduced to a singleton. Then the equation $P(x)=\tilde{f}(x)\pm \delta$ has at least $k+1$ solutions, say, $\alpha_0<\cdots<\alpha_k$. Since the polynomials $P(X)$ and $P(X)\pm \delta$ have the same leading coefficient, from (\ref{3}) and the fact that $P(X)$ has degree $<k$, for some $t_1,t_2\in [\alpha_0,\alpha_k]$ we get
\begin{align*}
0&=\frac{P^{(k)}(t_1)}{k!}=\sum_{j=0}^{k}\frac{P(\alpha_j)}{\prod_{0\leq i\leq k,i\neq j}(\alpha_j-\alpha_i)}=\sum_{j=0}^{k}\frac{P(\alpha_j)\pm \delta}{\prod_{0\leq i\leq k,i\neq j}(\alpha_j-\alpha_i)}\\
&=\sum_{j=0}^{k}\frac{\tilde{f}(\alpha_j)}{\prod_{0\leq i\leq k,i\neq j}(\alpha_j-\alpha_i)}=\frac{\tilde{f}^{(k)}(t_2)}{k!}\neq 0,
\end{align*}
which gives the contradiction.
\end{proof}
The above lemma gives a bound on the number of connected components and enables us to discuss which connected component contains most integral points. The following slight refinement of Definition \ref{def1} \citep[Definition 5.4]{huxley} would make this clear. Personally, I find it to be more intuitive.
\begin{defn}
Among the connected components of the set $\mathcal{C}_\delta\cap \mathcal{C}_P$, choose the one having the largest number of points $(n_{h+j},\lfloor f(n_{h+j})\rceil)$ for all $j\in \lbrace 1,\dots,l\rbrace$ with $l>k$. Then the set 
\begin{align*}
\bar{\mathcal{A}}=\lbrace n_{h+1},\dots,n_{h+l}\rbrace
\end{align*}
is called a proper major arc extracted from $\mathcal{A}$. The length of $\bar{\mathcal{A}}$ is the number
\begin{align*}
L=n_{h+l}-n_{h+1}.
\end{align*}
For convenience, we introduce the following numbers
\begin{align*}
a_k=36e^{-2}k(2e^3)^kc_k \quad\text{and}\quad b_k=20e^3k^2c_k^{\frac{2}{k(k-1)}},\tag{24}\label{24}
\end{align*}
so that $\beta_k=b_k+4k^2$, also $c_k$ comes from (\ref{1}). Notice that $a_k\geq 72e$ since $k,c_k\geq 1$.
\end{defn}
With all these definitions in hand, now we could deduce a few basic properties of the major arcs \citep[Lemma 5.13]{huxley}.
\begin{lem}\label{lem7}
Let $\mathcal{A}$ be a major arc associated to $f^{(k)}$ and $\bar{\mathcal{A}}$ be the proper major arc taken from $\mathcal{A}$ with denominator $q$, length $L$ and equation $y=P(x)$.
\begin{enumerate}
\item We have 
\begin{align*}
L\leq 2k\Big( \frac{\delta}{\lambda_k}\Big)^{\frac{1}{k}}.
\end{align*}
\item We have 
\begin{align*}
|\bar{\mathcal{A}}|\leq 2kLq^{-\frac{2}{k(k-1)}}.
\end{align*}
\item If $q\leq (a_k\delta)^{-1}$, then the distance $d$ between each point of $\mathcal{S}(f,N,\delta)\setminus\mathcal{A}$ and $\bar{A}$ satisfies
\begin{align*}
d>L(a_kq\delta)^{-\frac{1}{k}}.
\end{align*}
\end{enumerate}
\end{lem}
where $a_k$ is defined in (\ref{24}).
\begin{proof} \quad\\
\begin{enumerate}
\item Set $\bar{A}=\lbrace n_h,\dots,n_h+L\rbrace$ and define $g(x)=f(x)-P(x)$. Let $\alpha_j=n_h+j\frac{L}{k}$ for $j\in \lbrace 0,\dots,k\rbrace$. By using (\ref{3}) and the fact that $P(x)$ has degree $<k$ and by the assumption (\ref{1}) that $|f^{(k)}(x)|\asymp \lambda_k$, for some $t\in [\alpha_0,\alpha_k]$, we get
\begin{align*}
\frac{f^{(k)}(t)}{k!}=\frac{g^{(k)}(t)}{k!}&=\sum_{j=0}^{k}\frac{g(\alpha_j)}{\prod_{0\leq j\leq k,i\neq j}(\alpha_j-\alpha_i)}\\
&=\Big( \frac{k}{L}\Big)^k \sum_{j=0}^{k}\frac{g(\alpha_j)}{\prod_{0\leq i\leq k,i\neq j}(j-i)}.
\end{align*} 
We can see by the way we defined $g(x)$, $|g(\alpha_j)|\leq \delta$, the above equation can be rewritten as
\begin{align*}
\frac{\lambda_k}{k!}&\leq \frac{f^{(k)}(t)}{k!}\leq \delta \Big( \frac{k}{L}\Big)^k\sum_{j=0}^{k} \frac{1}{\prod_{0\leq i\leq k,i\neq j}(j-i)}\leq \delta \Big( \frac{k}{L}\Big)^k\sum_{j=0}^{k} \frac{1}{(k-j)!j!}\\
&=\frac{\delta}{k!}\Big( \frac{k}{L}\Big)^k\sum_{j=0}^{k} \binom{k}{j}=\frac{2^k\delta}{k!}\Big( \frac{k}{L}\Big)^k.
\end{align*}
A further note on the third inequality above, we can imagine the product $\prod_{0\leq i\leq k,i\neq j}(j-i)$ as multiplying the differences between $j$ and each element in $\lbrace 1,\dots,k \rbrace\setminus j$, to avoid the product being negative, we could simply separate the product to 2 parts, and get $|\prod_{0\leq i\leq k,i\neq j}(j-i)|=j!(k-j)!$. Also it still might be worth mentioning that $2^k=(1+1)^k=\sum_{j=0}^{k} \binom{k}{j}1^{k-j}1^j$, which gives us the final part above. By rearranging the above equation, we will obtain the required 
bound.

\item Let $n_1<\cdots<n_k$ be $k$ points lying in $\bar{\mathcal{A}}$. Using the formula given by the Lagrange interpolation polynomial that we have deduced in (\ref{2}) and (\ref{3}), since the degree of Lagrange polynomial of the major arc is $<k$, by the uniqueness of the Lagrange polynomial, we can express $P(x)$ as following
\begin{align*}
P(x)=\sum_{j=1}^{k}\Big( \prod_{i=1,i\neq j}^{k}  \frac{x-n_i}{n_j-n_i}\Big) \lfloor f(n_i)\rceil
\end{align*}
and
\begin{align*}
b_k=\sum_{j=1}^{k} \frac{\lfloor f(x_j)\rceil}{\prod_{1\leq i\leq k,i\neq j}(n_j-n_i)}
=\frac{A}{\prod_{1\leq i<j\leq k}(n_j-n_i)}.
\end{align*}
Where $b_k$ is the leading coefficient of $P(x)$ and A is an integer since all the data points we are interpolating here are integral points. If we expand (\ref{2}), we can see that for each coefficient of $P(x)$, it can be written in the form of having some integer as numerator and $\prod_{1\leq i<j\leq k}(n_j-n_i)$ as denominator, so $q$ must be something that is $\leq \prod_{1\leq i<j\leq k}(n_j-n_i)$ since it is the smallest positive integer such that $P\in \frac{1}{q}\mathbb{Z}[x]$. Hence, we have
\begin{align*}
q\leq \prod_{1\leq i<j\leq k}(n_j-n_i)\leq \prod_{1\leq i<j\leq k}(n_k-n_1)=(n_k-n_1)^{\frac{k(k-1)}{2}}.
\end{align*}
so that
\begin{align*}
L\geq n_k-n_1\geq q^{\frac{2}{k(k-1)}}.
\end{align*}
Which implies that for any such $k$ points, we have $\frac{L}{n_k-n_1}\leq Lq^{\frac{-2}{k(k-1)}}$. So it follows that there can only be at most $Lq^{\frac{-2}{k(k-1)}}$ distinct intervals in $\bar{A}$ with each contains $k$ distinct points. Thus
\begin{align*}
|\bar{A}|\leq k(\frac{L}{n_k-n_1}+1)\leq 2kLq^{-\frac{2}{k(k-1)}}.
\end{align*}
\item
For the third part of the proof, let $n\in \mathcal{S}(f,N,\delta)\setminus A$ and $n_0\in \bar{A}$. Without the loss of generality, we can assume that $n>n_0$ and let $m=\lfloor f(n) \rceil$. By the definition of major arc, we have $m\neq P(n)$ since otherwise $m$ should be in the maximal set $\mathcal{A}$. So we can say
\begin{align*}
|P(n)-m|\geq \frac{1}{q}\geq \frac{1}{3q}+2\delta.
\end{align*}
Because $qP(n)-qm\neq 0$ and $qP(n),qm\in \mathbb{Z}$ by definition, so $|q(P(n)-qm)|\geq 1$, also $q\leq (a_k\delta)^{-1}\leq (3\delta)^{-1}$ by assumption. Define $g(x)=f(x)-P(x)$, by the triangle inequality we get
\begin{align*}
|g(n)-g(n_0)|&\geq |g(n)|-|g(n_0)|\geq |g(n)|-\delta=|P(n)-f(n)|-\delta\\
&\geq |P(n)-m|-|f(n)-m|-\delta\geq \frac{1}{3q}+2\delta-\delta-\delta\\
&=\frac{1}{3q}.\label{25}\tag{25}
\end{align*}
Since $f\in C^k[N,2N]$, by Taylor's theorem with Lagrange reminder evaluated at $n_0$, we get
\begin{align*}
g(x)=\sum_{j=0}^{k-1}g^{(j)}(n_0)\frac{(x-n_0)^j}{j!}+R_{k}(x) \quad \text{where}\quad R_k(x)=\frac{g^{(k)}(t)}{k!}(x-n_0)^k.
\end{align*}
for some $t\in [n_0,n]$, thus
\begin{align*}
g(n)-g(n_0)=\sum_{j=1}^{k-1}g^{(j)}(n_0)\frac{d^j}{j!}+g^{(k)}(t)\frac{d^k}{k!}.
\end{align*}
Applying (\ref{1}) and (\ref{22}) on the function $g$ with $M_0=\delta$ and $M_k=c_k\lambda_k$, we get
\begin{align*}
|g(n)-g(n_0)|<4e\delta\sum_{j=1}^{k-1}\Big( \frac{ke}{j}\Big)^j\frac{d^j}{j!}\Big\lbrace k^{j+1}L^{-j}+e^{j-1}\Big( \frac{c_k\lambda_k}{\delta}\Big)^{j/k}\Big\rbrace+c_k\lambda_k\frac{d^k}{k!}. \label{26}\tag{26}
\end{align*}
From the first part of the proof, we know that $\Big( \lambda_k/\delta\Big)^{1/k} \leq 2k/L$, along with the inequality for $j\geq 1$, $j!>e(j/e)^j$ which can be proved by taking logarithms on both side. With all we just mentioned in hand, (\ref{26}) becomes
\begin{align*}
|g(n)-g(n_0)|<{}&4e\delta \sum_{j=1}^{k-1}\Big( \frac{ke}{j}\Big)^j \frac{1}{j!}\lbrace k^{j+1}+c_k^{j/k}e^{-1}(2ek)^j\rbrace (dL^{-1})^j+\frac{(2k)^kc_k\delta d^k}{k!L^k}\\
\leq{}& 4e\delta\sum_{j=1}^{k-1}\Big( \frac{ke}{j}\Big)^je^{-1}\Big( \frac{e}{j}\Big)^j\lbrace k^{j+1}+c_k^{j/k}e^{-1}(2ek)^j\rbrace(dL^{-1})^j\\
&+\frac{(2k)^kc_k\delta d^k}{L^k}e^{-1}\Big( \frac{e}{k}\Big)^k.
\end{align*}
Since for $j\in \lbrace 1,\dots,k-1\rbrace$, $(dL^{-1})^j\leq \max\Big(dL^{-1},(dL^{-1})^{k-1} \Big)$. We can continue to say
\begin{align*}
|g(n)-g(n_0)|<{}&4\delta\Big(dL^{-1}+(dL^{-1})^{k-1}\Big)\sum_{j=1}^{k-1}\Big( \frac{ke}{j}\Big)^{2j}\lbrace k+c_k^{j/k}e^{-1}(2e)^j \rbrace\\
&+(2e)^ke^{-1}c_k\delta(dL^{-1})^k\\
\leq{}& 4\delta(dL^{-1}+(dL^{-1})^{k-1})(k+c_k2^{k-1}e^{k-2})\sum_{j=1}^{k-1}\Big( \frac{ke}{j}\Big)^{2j}\\
&+(2e)^ke^{-1}c_k\delta(dL^{-1})^k.\\ \label{27}\tag{27}
\end{align*}
Notice the fact that $c_k2^{k-1}e^{k-2}\geq 2^{k-1}e^{k-2}$ since $c_k\geq1$, and by using induction for $k\in \mathbb{N}$ and $k\geq 3$, we can show $2^{k-1}>k$, thus $c_k2^{k-1}e^{k-2}>k$. So (\ref{27}) will become
\begin{align*}
|g(n)-g(n_0)|<2&^{k+2}e^{k-2}c_k\delta\Big(dL^{-1}+(dL^{-1})^{k-1} \Big)\sum_{j=1}^{k-1}\Big(\frac{ke}{j} \Big)^{2j}\\
&+(2e)^ke^{-1}c_k\delta(dL^{-1})^k.
\end{align*}
Using the technique we have discussed before, Lemma \ref{lem4}, also remembering that $a_k=36e^{-2}k(2e^3)^kc_k$ from (\ref{24}), we can get
\begin{align*}
|g(n)-g(n_0)|<{}&4e^{-2}k(2e^3)^kc_k\delta\Big( dL^{-1}+(dL^{-1})^{k-1} \Big)\\&+(2e)^ke^{-1}c_k\delta(dL^{-1})^k\\
={}&9^{-1}a_k\delta\Big(dL^{-1}+(dL^{-1})^{k-1}\Big)+(2e)^ke^{-1}c_k\delta(dL^{-1})^k\\
\leq{}& 9^{-1}a_k\delta\Big( dL^{-1}+(dL^{-1})^{k-1}+(dL^{-1})^k\Big).
\end{align*}
With the inequality we get from (\ref{25}), we have
\begin{align*}
&q^{-1}<3^{-1}a_k\delta\Big( dL^{-1}+(dL^{-1})^{k-1}+(dL^{-1})^k\Big)\\
&1<qa_k\delta\max\Big(dL^{-1},(dL^{-1})^{k-1},(dL^{-1})^k \Big).
\end{align*}
which means
\begin{align*}
dL^{-1}>\min \Big\lbrace (a_kq\delta)^{-1},(a_kq\delta)^{-\frac{1}{k-1}},(a_kq\delta)^{-\frac{1}{k}}\Big\rbrace.
\end{align*}
Since by the assumption, we have $a_k\delta q\leq 1$, so 
\begin{align*}
\min \Big\lbrace (a_kq\delta)^{-1},(a_kq\delta)^{-\frac{1}{k-1}},(a_kq\delta)^{-\frac{1}{k}}\Big\rbrace=(a_kq\delta)^{-\frac{1}{k}}.
\end{align*}
Which implies the statement of the lemma. \qedhere
\end{enumerate}
\end{proof}
A reasonable question to ask right now is what is the total number of points in $\mathcal{R}(f,N,\delta)$ that can be contained in some major arc associated to $f^{(k)}$? Remember that by the definition, given different Lagrange polynomials, there could be many major arcs associated to $f^{(k)}$. We now are in the position to estimate such contribution of points coming from major arcs. Proof of the lemma is changed from the original paper \citep{huxley1995points} and the book \citep[Lemma 5.14]{huxley} due to the oversight we have mentioned in the introduction, if reader is interested, have a look at the original proof which can be found both in the paper and the book would be a very good idea.
\begin{lem} \label{lem8}
Let $R_0$ be the contribution of points coming from the major arcs associated to $f^{(k)}$, such that $R_0\subset\mathcal{R}(f,N,\delta)$. Then
\begin{itemize}
\item If the set of points in each proper major arc with denominator $\leq (a_k\delta)^{-1}$ can be covered with an interval of length $L_j(a_kq_j\delta)^{-1/k}$ which is disjoint from all such covers for other major arcs, where $L_j$ and $q_j$ are the length and denominator associated with their proper major arcs:
\begin{align*}
R_0\leq b_kN\delta^{\frac{2}{k(k-1)}}+8k^3\Big( \frac{\delta}{\lambda_k}\Big)^{1/k}+10e^3k^2.
\end{align*}
\item Otherwise:
\begin{align*}
R_0\leq \frac{3}{2} b_kN\delta^{\frac{2}{k(k-1)}}+16k^3\Big( \frac{\delta}{\lambda_k}\Big)^{1/k}+10e^3k^2
\end{align*}
where $a_k,b_k$ are defined in (\ref{24}).

\end{itemize}
\end{lem}
\begin{proof}
Let $\mathcal{M}_0$ be the set of major arcs and $Q_k>0$ be the real number defined by
\begin{align*}
Q_k=(a_k\delta)^{-1}
\end{align*}
where $a_k$ is defined in (\ref{24}). Write $\mathcal{M}_0=\mathcal{M}_1\cup\mathcal{M}_2$ where $\mathcal{M}_1$ is the set of major arcs with denominator $>Q_k$ and $\mathcal{M}_2=\mathcal{M}_0\setminus \mathcal{M}_1$. For $i\in \lbrace 1,2\rbrace$, let 
\begin{align*}
S_i=\bigcup_{\mathcal{A}\in \mathcal{M}_i}\bar{\mathcal{A}} \quad \text{and} \quad R_i=|S_i|.
\end{align*}
We can see that $R_i$ can be understood as the total number of points contained in the proper major arcs inside $\mathcal{M}_i$. From Lemma \ref{lem5}, we know that there are at most $k$ connected components. By definition, the proper major arc is the component that intersects the most integral points, so we can deduce $R_0\leq k(R_1+R_2)$.
\begin{itemize}
\item \textbf{Estimate of $R_1$}. This estimate is rendered trivial by the second part of the result of Lemma \ref{lem7}. Choose a proper major arc in $\mathcal{M}_1$ with $k$ consecutive points $n_1,\dots,n_k$ such that $n_k-n_1$ is the smallest. By the Lemma, we have $n_k-n_1\geq q^{\frac{2}{k(k-1)}}>Q_k^{\frac{2}{k(k-1)}}$, so that
\begin{align*}
R_1\leq k(NQ_k^{-\frac{2}{k(k-1)}}+1)<{}&10e^3k\Big( N(c_k\delta)^{\frac{2}{k(k-1)}}+1\Big)\\
={}&b_k(2k)^{-1}N\delta^{\frac{2}{k(k-1)}}+10e^3k.
\end{align*}
\item \textbf{Estimate of $R_2$}. Since $q$ is a positive integer, we may assume that $Q_k\geq 1$, otherwise $S_2=\emptyset$ which is just trivial. Let $\bar{\mathcal{A}_1},\dots,\bar{\mathcal{A}}_J$ be the ordered sequence of proper major arcs with denominator $q_j\leq Q_k$. For each proper major arc $\bar{\mathcal{A}}_J$, we set $n_j$ and $L_j$ its first point and length, and define
\begin{align*}
d_j=L_j(a_kq_j\delta)^{-1/k}
\end{align*}
and $\mathcal{I}_j=[n_j,n_j+d_j]$.

To continue, we first need to justify the fact that two proper major arcs with their denominators $q_j,q_i\leq Q_k$ do not intersect each other. Notice that $\bar{\mathcal{A}}_i \nsubseteq \mathcal{A}_j$, this is indeed true by the definition of the major arc and the uniqueness of the Lagrange interpolation. Since a proper major arc must contain more than $k+1$ points, the uniqueness of the Lagrange interpolation tells us there is only one such polynomial with degree $<k$, if $\bar{\mathcal{A}}_i \subset \mathcal{A}_j$, it will result in the two Lagrange interpolation polynomials associate to the two major arcs equal being $P_j=P_i$. So there must exist a point $a$ such that $a\in \bar{\mathcal{A}}_i$ and $a\in \mathcal{S}(f,N,\delta)\setminus \mathcal{A}_j$, by Lemma \ref{lem7} we have $\operatorname{dist}(a, \bar{\mathcal{A}}_j)>d_j $. Similarly, $\exists$ $b\in \bar{\mathcal{A}}_j$ such that $\operatorname{dist} (b,\bar{\mathcal{A}}_i)>d_i$. Now, let's assume two proper major arcs with their denominators $q_j,q_i\leq Q_k$ intersect each other. Without loss of generality, let's assume that the first point of one proper major arc intersects the last point of the other proper major arc. As showing in the graph below.

\begin{figure}[H]
\centering
\includegraphics[scale=0.45]{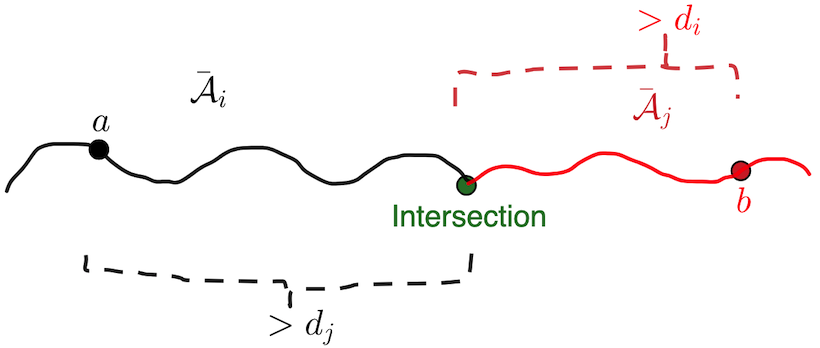}
\caption{Intersection of two proper major arcs}
\label{fig2}
\end{figure}

In the case of intersection, since $q_i,q_j\leq Q_k$, we have 
\begin{align*}
d_j\geq L_j(a_kQ_k\delta)^{-1/k}=L_j \quad \text{and} \quad d_i\geq L_i(a_kQ_k\delta)^{-1/k}=L_i
\end{align*}
And this will give us a contradiction as follows
\begin{align*}
L_i>d_j\geq L_j \quad \text{and} \quad L_j>d_i\geq L_i
\end{align*}
So two proper major arcs in the setting we are in cannot have an intersection.
\begin{case} Distance between $n_j$ and $n_{j+1}$ $>d_j$.

We can see that $\mathcal{I}_j$ contains $\bar{\mathcal{A}_j}$ and does not contain $\bar{\mathcal{A}}_{j+1}$, this is indeed true, since $q_j\leq Q_k$, we have $d_j\geq L_j(a_kQ_k\delta)^{-1/k}=L_j$. Also, by assumption, $\mathcal{I}_j\cap \bar{\mathcal{A}}_{j+1}=\emptyset$.

Therefore the intervals $\mathcal{I}_j$ are pairwise distinct. Using Lemma {\ref{lem3}} with $S=S_2$, we get
\begin{align*}
R_2\leq N\max_{j}\frac{|\bar{\mathcal{A}}_j|}{d_j}+2\max_j|\bar{\mathcal{A}}_j| \label{28} \tag{28} .
\end{align*}

Now by second part of Lemma \ref{lem7} and the choice of $d_j$, we have
\begin{align*}
\frac{|\bar{\mathcal{A}}_j|}{d_j}\leq 2kL_jq_j^{-\frac{2}{k(k-1)}}L_j^{-1}(a_kq_j\delta)^{1/k}=2k(a_k\delta)^{1/k}q_j^{\frac{k-3}{k(k-1)}}
\end{align*}
and since $q_j\leq Q_k=(a_k\delta)^{-1}$ and $k\geq 3$, we obtain
\begin{align*}
\frac{|\bar{\mathcal{A}}_j|}{d_j}\leq 2k(a_k\delta)^{\frac{2}{k(k-1)}}<10e^3k(c_k\delta)^{\frac{2}{k(k-1)}}=b_k(2k)^{-1}\delta^{\frac{2}{k(k-1)}}
\end{align*}
for $j\in \lbrace 1,\dots,J\rbrace$, and therefore by part 1 and 2 of Lemma \ref{lem7}
\begin{align*}
R_2\leq b_k(2k)^{-1}N\delta^{\frac{2}{k(k-1)}}+4k\max_{j}L_j\leq b_k(2k)^{-1}N\delta^{\frac{2}{k(k-1)}}+8k^2\Big( \frac{\delta}{\lambda_k}\Big)^{1/k}.
\end{align*}
Combining both estimates and the fact that $R_0\leq k(R_1+R_2)$, we get the result.
\end{case}
\begin{case} Distance between $n_j$ and $n_{j+1}$ $\leq d_j$.

Here we can see that the intersection between $\mathcal{I}_j\cap \bar{\mathcal{A}}_{j+1}\neq \emptyset$, the graph below should give a better idea.
\begin{figure}[H]
\centering
\includegraphics[scale=0.45]{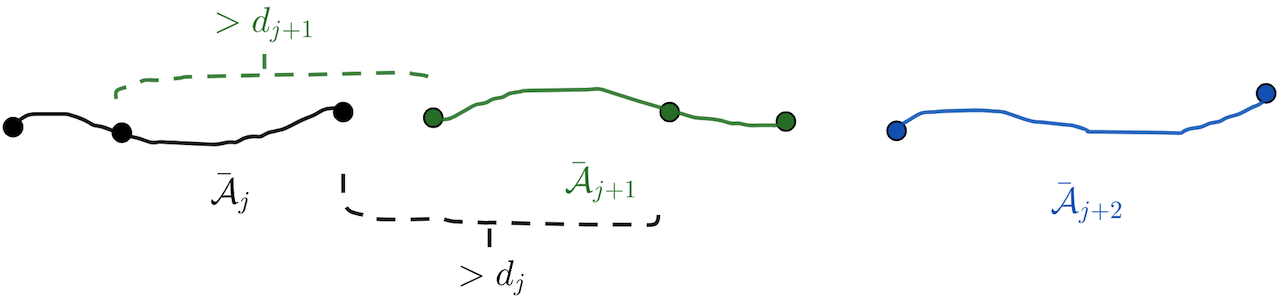}
\caption{Proper major arcs}
\label{fig3}
\end{figure}
Since we have proved that proper major arcs are disjoint from each other. By part three of Lemma \ref{lem7}, there exist a point in $\bar{\mathcal{A}}_{j+1}$ with distance $>d_j$ away from any point in $\bar{\mathcal{A}}_j$. We can see that $\mathcal{I}_j$ must be disjoint from $\bar{\mathcal{A}}_{j+2}$. Also $\bar{\mathcal{A}}_j \subset \mathcal{I}_j$ as in $Case$ 1. Let
\begin{align*}
S_3=\bigcup_{i \text{ is odd}}\bar{\mathcal{A}}_i \quad \text{and}\quad S_4=\bigcup_{i \text{ is even}}\bar{\mathcal{A}}_i \quad \text{and}\quad R_j=|S_j|,\quad j\in \lbrace 3,4\rbrace.
\end{align*}
We can see $S_3 \cup S_4=S_2$ and $R_2=R_3+R_4$. By the same argument as $Case$ 1, we have
\begin{align*}
R_3,R_4\leq b_k(2k)^{-1}N\delta^{\frac{2}{k-1}}+8k^2\Big( \frac{\delta}{\lambda_k}\Big)^{1/k}.
\end{align*}
So
\begin{align*}
R_2\leq b_k k^{-1}N\delta^{\frac{2}{k(k-1)}}+16k^2\Big( \frac{\delta}{\lambda_k}\Big)^{1/k}.
\end{align*}
Combining both estimate, we get
\begin{align*}
R_0\leq k(R_1+R_2)\leq \frac{3}{2}b_kN\delta^{\frac{2}{k(k-1)}}+16k^3\Big( \frac{\delta}{\lambda_k}\Big)^{1/k}+10e^3k^2. \quad \qedhere
\end{align*}
\end{case}
\end{itemize}
\end{proof}
\subsection{The Proof of the Theorem of Huxley and Sargos} 
From the end of the previous section, we get an estimation of the contribution of points of $\mathcal{S}(f,N,\delta)$ which come from major arcs. The next natural step to progress our proof of the theorem of Huxley and Sargos is to estimate the contribution of those points which do not come from major arcs. The next Lemma \citep[Lemma 5.15]{huxley} we will state provides such an estimate, and its proof is similar to the proof of the $k$th derivative test. Also, again to avoid the trivial case, we may assume $\mathcal{R}(f,N,\delta)\geq k+1$.
\begin{lem} \label{lem9}
Let $N\leq n_0<\cdots<n_k\leq 2N$ be $k+1$ points of $\mathcal{S}(f,N,\delta)$, which do not lie on the same algebraic curve of degree $<k$. Then
\begin{align*}
n_k-n_0>\min\Big( (c_k\lambda_k)^{-\frac{2}{k(k+1)}}, 2^{-1}\delta^{-\frac{2}{k(k-1)}}\Big).
\end{align*}
\end{lem}
\begin{proof}
By definition, there exist integers $m_0,\dots,m_k$ and real numbers $\delta_0,\dots,\delta_k$ such that
\begin{align*}
f(n_j)=m_j+\delta_j
\end{align*}
with $|\delta_j|<\delta$ for all $j\in \lbrace 0,\dots,k\rbrace$. We define $D_k$ as follows
\begin{align*}
D_k=\prod_{0\leq h<i\leq k}(n_i-n_h)>0.
\end{align*}
and let $P=b_kX^k+\cdots+b_0$ be the Lagrange polynomial interpolating the points $(n_j,m_j)$. Then we have
\begin{align*}
b_k=\sum_{j=0}^{k}\frac{m_j}{\prod_{0\leq i\leq k,i\neq j}(n_j-n_i)}=\frac{A_k}{D_k}
\end{align*}
where $A_k\in \mathbb{Z}$, since the points we are interpolating are integral points. Reasoning exactly in the same way as in the proof of Theorem \ref{thm5}, we get
\begin{align*}
k!A_k=D_kf^{(k)}(t)-k!D_k\sum_{j=0}^{k}\frac{\delta_j}{\prod_{0\leq i\leq k,i\neq j}(n_j-n_i)}.
\end{align*}
Since the points we are interpolating do not lie on the same algebraic curve of degree $<k$, we have $b_k\neq 0$. Since $A_k$ is an integer, we have $|A_k|\geq 1$, and using $|\delta_j|<\delta$ and condition (\ref{1}) we get
\begin{align*}
k!\leq k!|A_k|&<c_k\lambda_k D_k+k!\delta D_k\sum_{j=0}^{k}\frac{1}{\prod_{0\leq i\leq k,i\neq j}|n_j-n_i|}\\
&=c_k\lambda_k D_k+k!\delta\sum_{j=0}^{k}\prod_{\substack{0\leq h<i\leq k\\ h\neq j,i\neq j}}(n_i-n_h)\\
&\leq c_k\lambda_k(n_k-n_0)^{\frac{k(k+1)}{2}}+(k+1)!\delta(n_k-n_0)^{\frac{k(k-1)}{2}}.
\end{align*}
which suggests that either
\begin{align*}
k!\leq 2c_k\lambda_k(n_k-n_0)^{\frac{k(k+1)}{2}} \quad \text{or} \quad k!\leq 2(k+1)!\delta(n_k-n_0)^{\frac{k(k-1)}{2}}
\end{align*}
and this implies 
\begin{align*}
n_k-n_0>\min\Big( \Big(\frac{k!}{2} \Big)^{\frac{2}{k(k+1)}}(c_k\lambda_k)^{-\frac{2}{k(k+1)}}, (2k+2)^{-\frac{2}{k(k-1)}}\delta^{-\frac{2}{k(k-1)}}  \Big) \quad \qedhere
\end{align*}
\end{proof}
Finally, with all the techniques and results in hand, we are in the position to give a conclusion to this chapter, which is the proof of the theorem of Huxley and Sargos Theorem \ref{thm1}. Beside changes due to the oversight, the general spirit of the proof stays true to the original. The original proof was given both in the book by Bordell\`es \citep[p. 383]{huxley} and the paper by Huxley \citep{huxley1995points}.
\begin{proof}
Let $S_0$ be the set of points of $\mathcal{S}(f,N,\delta)$ coming from the major arcs and $T_0=\mathcal{S}(f,N,\delta)\setminus S_0$. By Lemma \ref{lem8}, we have
\begin{align*}
|S_0|=R_0\leq b_kN\delta^{\frac{2}{k(k-1)}}+8k^3\Big( \frac{\delta}{\lambda_k}\Big)^{1/k}+10e^3k^2
\end{align*}
or depending on the condition we discussed before
\begin{align*}
|S_0|=R_0\leq \frac{3}{2} b_kN\delta^{\frac{2}{k(k-1)}}+16k^3\Big( \frac{\delta}{\lambda_k}\Big)^{1/k}+10e^3k^2
\end{align*}
where $b_k$ is given in (\ref{24}). Now let $G=\lbrace n_0,\dots,n_{k^2}\rbrace$ be a set of $k^2+1$ consecutive ordered points of $T_0$, again here we assume the size of $T_0$ is larger than $k^2+1$, otherwise the result is trivial. Since $G$ is not contained in any major arc, so we can find an integer $j\in\lbrace k,\dots,k^2 \rbrace$ such that $j+1$ points $(n_i,m_i)$ do not lie on the same algebraic curve of degree $<k$. By Lemma \ref{lem9}, we have
\begin{align*}
n_{k^2}-n_0\geq n_j-n_0\geq n_k-n_0> \min \Big(  (c_k\lambda_k)^{-\frac{2}{k(k+1)}},2^{-1}\delta^{-\frac{2}{k(k-1)}}  \Big)
\end{align*}
which means
\begin{align*}
|T_0|\leq 2k^2\Big(N(c_k\lambda_k)^{\frac{2}{k(k+1)}}+ 2N\delta^{\frac{2}{k(k-1)}}+1   \Big)
\end{align*}
and the proof is completed with the fact that $\mathcal{R}(f,N,\delta)\leq |T_0|+|S_0|$.
\end{proof}
\section{Optimality of The Methods}
\subsection{Integral Points Around Polynomials} 
In this chapter, we will discuss the optimality of Theorem \ref{thm1}, that is whether the bound
\begin{align*}
\mathcal{R}(f,N,\delta)\ll N\lambda_k^{\frac{2}{k(k+1)}}+N\delta^{\frac{2}{k(k-1)}}+\big( \frac{\delta}{\lambda_k}\big)^\frac{1}{k}+1 \label{29}\tag{29}
\end{align*}
could be improved with some extra conditions or if the bound is false without certain conditions. A good place to start will be putting the theorem to the test where polynomials are concerned. First, let's discuss some general scenario where a better bound can be obtained.
\begin{prop}
Let $\delta\in [0,\frac{1}{4}]$, $N,k\in \mathbb{Z}_{\geq 1}$ and $P=\alpha_kX^k+\dots+\alpha_0\in \mathbb{R}[x]$. If there exist $a\in \mathbb{Z}$ and $q\in [1,N^k]\cap \mathbb{Z}$ such that the greatest common divisor $(a,q)=1$ and $|\alpha_k-\frac{a}{q}|\leq \frac{1}{q^2}$, then for all $\epsilon>0$
\begin{align*}
\mathcal{R}(P,N,\delta) \ll Nq^{-1/k}+N\delta^{\frac{2}{k(k+1)}}+N^\epsilon
\end{align*}
the term $N^\epsilon$ being replaced by $1$ if $k=1$.
\end{prop}
If $\alpha_k\in (0,1]$, then we can take $a=1$ and $q=\lfloor\alpha_k^{-1}\rfloor$. So if we further assume that $\alpha_k\in [N^{-k},1]$, then we derive
\begin{align*}
\mathcal{R}(P,N,\delta) \ll N\alpha_k^{1/k}+N\delta^{\frac{2}{k(k+1)}}+N^\epsilon. \label{30} \tag{30}
\end{align*}
If $\alpha_k>1$, then the bound (\ref{30}) is trivial provided that $\alpha_k\geq N^{-k}$. Now if $\alpha_k<N^{-k}$, then $N<\alpha^{-1/k}$, and thus we can derive the following corollary.
\begin{cor}
Let $\delta\in [0,\frac{1}{4}]$, $N,k\in \mathbb{Z}_{\geq 1}$ and $P=\alpha_kX^k+\alpha_{k-1}X^{k-1}+\dots+\alpha_0\in \mathbb{R}[x]$ with $\alpha_k>0$. Then, for all $\epsilon>0$,
\begin{align*}
\mathcal{R}(P,N,\delta)\ll N\alpha_k^{1/k}+N\delta^{\frac{2}{k(k+1)}}+\alpha_k^{-1/k}+N^\epsilon,
\end{align*}
the term $N^\epsilon$ being replaced by $1$ if $k=1$. 
\end{cor}
Next, we can see from the proof of Theorem \ref{thm1} that the bound $\mathcal{R}(f,N,\delta)\leq |T_0|+|S_0|$ has two parts. If the function $f$ is already in the form of a Lagrange polynomial with rational coefficient, then the contribution of the points in the set $\mathcal{S}(f,N,\delta)$ will only come from major arcs, that is
\begin{align*}
\mathcal{R}(f,N,\delta)\leq b_kN\delta^{\frac{2}{k(k-1)}}+8k^3\Big( \frac{\delta}{\lambda_k}\Big)^{1/k}+10e^3k^2
\end{align*}
or 
\begin{align*}
\mathcal{R}(f,N,\delta)\leq \frac{3}{2} b_kN\delta^{\frac{2}{k(k-1)}}+16k^3\Big( \frac{\delta}{\lambda_k}\Big)^{1/k}+10e^3k^2.
\end{align*}
\subsection{Further Refinements}
The refinements of Theorem \ref{thm1} have been made in several directions, first we will see that the result can still hold for $k=2$ \citep{branton1994points}. We will show the following version which is stated explicitly in \citep[Theorem 5.6]{huxley}.
\begin{thm}[Branton-Sargos]\label{thm10} Let $f\in C^2[N,2N]$ such that there exists $\lambda_2>0$ such that, for all $x\in [N,2N]$, we have
\begin{align*}
|f''(x)|\asymp \lambda_2.
\end{align*}
Then
\begin{align*}
\mathcal{R}(f,N,\delta)\ll N\lambda_2^{1/3}+N\delta+\Big( \frac{\delta}{\lambda_2}\Big)^{1/2}+1.
\end{align*}
If in addition there exists $\lambda_1>0$ such that $|f'(x)|\asymp \lambda_1$ for all $x\in [N,2N]$, then
\begin{align*}
\mathcal{R}(f,N,\delta) \ll N\lambda_2^{1/3}+N\delta+\lambda_1\Big( \frac{\delta}{\lambda_2}\Big)^{1/2}+\frac{\delta}{\lambda_1}+1.
\end{align*}
\end{thm}
From Theorem \ref{thm1}, $N\lambda_k^{\frac{2}{k(k+1)}}$ is the main term of the equality. It's very difficult to improve on the main term. The others are the secondary terms and the quantity $(\delta\lambda_k^{-1})^{1/k}$ is quasi-optimal. So we turn our eyes on the term $N\delta^{\frac{2}{k(k-1)}}$, to see whether it may be improved. This is done by generalizing the method used in Theorem \ref{thm10} and by using a $k$-dimensional version of the reduction principle and a new divisibility relation on the divided differences discovered by Filaseta and Trifonov. They proved the following result \citep[Theorem 5.7]{huxley}.
\begin{thm}
Let $k\geq 3$ be an integer and $f\in C^k[N,2N]$ such that there exist $\lambda_{k-1}>0$ and $\lambda_k>0$ such that, for all $x\in [N,2N]$, we have
\begin{align*}
|f^{(k-1)}(x)|\asymp \lambda_{k-1},\quad |f^{(k)}(x)|\asymp \lambda_{k}\quad \text{and} \quad \lambda_{k-1}=N\lambda_k. \tag{31}\label{31}
\end{align*}
Then the following upper bounds hold.
\begin{itemize}
\item For all $k\geq 3$, we have
\begin{align*}
\mathcal{R}(f,N,\delta)\ll N\lambda_k^{\frac{2}{k(k+1)}}+N\delta^{\frac{2}{(k-1)(k-2)}}+N(\delta\lambda_{k-1})^{\frac{2}{k^2-k+2}}+\Big( \frac{\delta}{\lambda_{k-1}}\Big)^{\frac{1}{k-1}}+1.
\end{align*}
\item For $k=3$, we have
\begin{align*}
\mathcal{R}(f,N,\delta)\ll N\lambda_3^{1/6}+N\delta^{2/3}+N(\delta^3\lambda_3)^{1/12}+\Big(\frac{\delta}{\lambda_2} \Big)^{1/2}+1.
\end{align*}
\item For all $k\geq 4$ and $\epsilon>0$, we have
\begin{align*}
\mathcal{R}(f,N,\delta)\ll \Big\lbrace &N\lambda_k^{\frac{2}{k(k+1)}}+N(\delta\lambda_{k-1})^{\frac{2}{k^2-k+2}}  +N\delta^{\frac{4}{k^2-3k+6}}\\
&+N(\delta^2N^{-1}\lambda_{k-1}^{-1})^{\frac{2}{k^2-3k+4}} \Big\rbrace N^\epsilon+ \Big(  \frac{\delta}{\lambda_{k-1}}\Big)^{\frac{1}{k-1}}+1.
\end{align*}
\item For all $k\geq 5$, we have
\begin{align*}
\mathcal{R}(f,N,\delta)\ll N\lambda_k^{\frac{2}{k(k+1)}}+ N\delta^{\frac{2}{(k-1)(k-2)}}+\Big(  \frac{\delta}{\lambda_{k-1}}\Big)^{\frac{1}{k-1}}+1.
\end{align*}
\end{itemize}
\end{thm}
We can see easily that Theorem \ref{thm1} implies the $k$th derivative test, because the condition $\delta\ll \lambda_k$ means the main term dominates all the others. The following result is an estimate analogous to the $k$th derivative test but with a more flexible condition than (\ref{18}) \citep[Proposition 5.2]{huxley}.
\begin{prop} \label{prop1}
Let $f\in C^{\infty}[N,2N]$ such that there exists $T\geq 1$ such that, for all $x\in [N,2N]$ and all $j\in \mathbb{Z}_{\geq 0}$, we have
\begin{align*}
|f^{(j)}(x)|\asymp \frac{T}{N^j},  \tag{32}\label{32}
\end{align*}
and 
\begin{align*}
N\delta\leq T\leq \delta^{-1}.  \tag{33}\label{33}
\end{align*}
Then for all $k\geq 1$, we have
\begin{align*}
\mathcal{R}(f,N,\delta)\ll T^{\frac{2}{k(k+1)}}N^{\frac{k-1}{k+1}}.
\end{align*}
\end{prop}
Note that, using the notion of (\ref{32}), for $k\geq 2$, Huxley and Sargos's result may be stated as follow
\begin{align*}
\mathcal{R}(f,N,\delta)\ll T^{\frac{2}{k(k+1)}}N^{\frac{k-1}{k+1}}+N\delta^{\frac{2}{k(k-1)}}+N(\delta T^{-1})^{1/k}, \tag{34}\label{34}
\end{align*}
and the term $N(\delta T^{-1})^{1/k}$ is dominated by the term $N\delta^{\frac{2}{k(k-1)}}$ for $k\geq 3$. So if Proposition \ref{prop1} is true, it tells us that the condition (\ref{33}) is sufficient to remove both these two terms mentioned.
\begin{proof}
We will be using induction to prove the proposition. For $k=1$, by the first derivative test we have
\begin{align*}
\mathcal{R}(f,N,\delta)&\ll N\lambda_1+N\delta+\frac{\delta}{\lambda_1}+1= N\frac{T}{N}+N\delta+\frac{\delta}{T/N}+1\\
&=T+N\delta+\delta\frac{N}{T}+1\ll T \qquad\text{by (\ref{33})}.
\end{align*} 
Now suppose the claim to be true for some $k\geq 1$, that is $\mathcal{R}(f,N,\delta) \ll T^{\frac{2}{k(k+1)}}N^{\frac{k-1}{k+1}}$. Now by Huxley and Sargos's result (\ref{34}), what we have for $k+1$ is
\begin{align*}
\mathcal{R}(f,N,\delta)\ll T^{\frac{2}{(k+1)(k+2)}}N^{\frac{k}{k+2}}+N\delta^{\frac{2}{k(k+1)}}+N(\delta T^{-1})^{\frac{1}{k+1}}.
\end{align*}
Let
\begin{align*}
E=\max \Big(T^{\frac{2}{(k+1)(k+2)}}N^{\frac{k}{k+2}},N\delta^{\frac{2}{k(k+1)}}, N(\delta T^{-1})^{\frac{1}{k+1}}\Big)=\max (e_1,e_2,e_3).
\end{align*}
We will get
\begin{align*}
\mathcal{R}(f,N,\delta)\ll \min \Big( E,T^{\frac{2}{k(k+1)}}N^{\frac{k-1}{k+1}}  \Big).
\end{align*}
The result follows if $E=e_1$. Now to discuss the cases where $E=e_2$ and $E=e_3$ separately we will need the following obvious inequality. If $x,y\geq 0$ and $0\leq a \leq 1$, then $\min (x,y)\leq x^ay^{1-a}$.
\begin{itemize}
\item Case $E=e_2$, we choose $a=\frac{1}{k+2}$. By the inequality we have deduced and (\ref{33})
\begin{align*}
\min\Big(e_2, T^{\frac{2}{k(k+1)}}N^{\frac{k-1}{k+1}}\Big)&\leq T^{\frac{2}{(k+1)(k+2)}}N^{\frac{k}{k+2}}(T\delta)^{\frac{2}{k(k+1)(k+2)}}\\
&\leq T^{\frac{2}{(k+1)(k+2)}}N^{\frac{k}{k+2}}.
\end{align*}
\item Case $E=e_3$, we choose $a=\frac{2}{k+2}$, similar as in the previous case, we get
\begin{align*}
\min \Big(  e_3,T^{\frac{2}{k(k+1)}}N^{\frac{k-1}{k+1}}\Big)&\leq T^{\frac{2}{(k+1)(k+2)}}N^{\frac{k}{k+2}}(N\delta T^{-1})^{\frac{2}{(k+1)(k+2)}}\\
&\leq T^{\frac{2}{(k+1)(k+2)}}N^{\frac{k}{k+2}}.
\end{align*}
Which completes the proof. \quad \qedhere
\end{itemize}
\end{proof}
The main term in the case $k=2$ was improved by Huxley \citep{huxley1996area}, Huxley and Trifonov \citep{hss} and then Trifonov \citep{trifonov2002lattice} again, who extended an earlier work by Swinnerton Dyer. The following result is one of many versions of the theorem which was proved \citep[Theorem 5.8]{huxley}.
\begin{thm}[Huxley]\label{thm11} Let $f\in C^3[N,2N]$ such that there exist $C\geq 1$, $0<\lambda_2\leq C^{-1}$ and $\lambda_3>0$ such that for all $x\in [N,2N]$, we have
\begin{align*}
C^{-1}\lambda_2\leq |f''(x)|\leq C\lambda_2, \quad C^{-1}\lambda_3\leq |f'''(x)|\leq C\lambda_3\quad \text{and}\quad \lambda_2=N\lambda_3. \tag{35}\label{35}
\end{align*}
Then
\begin{align*}
\mathcal{R}(f,N,\delta)\ll& \Big\lbrace N^{9/10}\lambda_2^{3/10}+N^{4/5}\lambda_2^{1/5}+N\lambda_2^{3/8}\delta^{1/8}+N^{7/8}\lambda_2^{1/4}\delta^{1/8}\\
&+N^{6/7}(\lambda_2\delta)^{1/7}+N\lambda_2^{1/5}\delta^{2/5} \Big\rbrace (\log N)^{2/5}+N\delta\\
&+(\delta\lambda_2^{-1})^{1/2}+1.
\end{align*}
The implied constant depends only on $C$.
\end{thm}
Finally for the sake of completion, we will show the following slight improvement obtained due to Trifonov \citep{trifonov2002lattice} and the explicit statement can also be found in \citep[Theorem 5.9]{huxley}.
\begin{thm}[Trifonov]
Let $f\in C^3[N,2N]$ such that there exist $C\geq 1$, $0<\lambda_2\leq 1$ and $\lambda_3>0$ such that for all $x\in [N,2N]$, we have
\begin{align*}
C^{-1}\lambda_2\leq |f''(x)|\leq C\lambda_2, \quad C^{-1}\lambda_3\leq |f'''(x)|\leq C\lambda_3 \quad \text{and} \quad \lambda_2=N\lambda_3 \tag{36}\label{36}
\end{align*}
and
\begin{align*}
N\lambda_2\geq 1 \quad \text{and}\quad N\delta^2\leq C^{-1} \tag{37}\label{37}
\end{align*}
Then for all $\epsilon>0$, we have
\begin{align*}
\mathcal{R}(f,N,\delta)\ll& \Big\lbrace N^{43/54}\lambda_2^{4/27}+N^{4/5}\lambda_2^{4/25}+N^{9/10}\delta^{4/15}+N^{12/13}\delta^{4/13}\\
&+N^{6/7}\lambda_2^{2/7}+N\lambda_2+N(\lambda_2\delta)^{1/4}      \Big\rbrace N^{\epsilon}+\lambda_2(N\delta)^{5/2}.
\end{align*}
The implied constant depends only on $C$ and $\epsilon$.
\end{thm}
The above two improvements were done on the main term of the inequality we get in Theorem \ref{thm1}. As we can see in the case where $k=3$, the main term in the Theorem \ref{thm1} is $N\lambda^{1/6}$. The above main terms are all smaller than that. Which will give a smaller upper bound for the number of integer points we are trying to estimate when the main terms dominate.
\subsection{Smooth Curves}
Simply by using the first derivative test and letting $\delta\rightarrow 0$, we get deduce a bound of the number of integer points lying on the arc of the curve $y=f(x)$ with $N<x\leq 2N$, which is $\ll N\lambda_k^{\frac{2}{k(k+1)}}+1$. Historically, this number was first investigated by Jarník \citep{jar} who proved that a strictly convex arc $y=f(x)$ with length $L$ has at most $\leq \frac{3}{(2\pi)^{1/3}}L^{2/3}+O(L^{1/3})$ integer points on the arc of the curve, this is a nearly best possible result under the sole condition of convexity.

However, Swinnerton Dyer \citep{swin} and Schmidt \citep{schmidt} proved independently that if $f\in C^3[0,N]$ is such that $|f(x)|\leq N$ and $f'''(x)\neq 0$ for all $x\in [0,N]$, then the number of integer points on the arc $y=f(x)$ with $0\leq x\leq N$ is $\ll N^{3/5+\epsilon}$. This result was then generalized by Bombieri and Pila \citep{bomb} as follows \citep[Proposition 5.3]{huxley}.

\begin{prop}[Bombieri-Pila]
Let $N\geq 1$, $k\geq 4$ be integers and set $K=\binom{k+2}{2}$. Let $\mathcal{I}$ be an interval with length $N$ and $f\in C^k(\mathcal{I})$ satisfying $|f'(x)|\leq 1$, $f''(x)>0$ and such that the number of solutions of the equation $f^{(K)}(x)=0$ is $\leq m$. Then there exists a constant $c_0=c_0(k)>0$ such that the number of integer points on the arc $y=f(x)$ with $x\in \mathcal{I}$ is 
\begin{align*}
\leq c_0(m+1)N^{1/2+3/(k+3)}
\end{align*}
\end{prop}

The ideas of Bombieri and Pila have been extended by Huxley \citep{huxley2007integer} to apply on the problems involving counting the number of integer points which are very close to regular curves, an explicit statement can also be found in \citep[Proposition 5.4]{huxley}. The function is supposed to be $C^5$ and along with the usual non-vanishing conditions of the derivatives on $[N,2N]$, the proof also requires lower bounds of the following determinants:
\begin{align*}
D_1(f;x)=\begin{vmatrix}
f'''(x) & 3f''(x)\\
f^{(4)}(x) & 4f'''(x)\\
\end{vmatrix}\begin{matrix} \vphantom{x}\\ \vphantom{y} .\end{matrix}
\end{align*}
\begin{align*}
D_2(f;x)=\frac{1}{2f''(x)}\begin{vmatrix}
f'''(x) & 3f''(x) & 0\\
f^{(4)}(x) & 4f'''(x) & 6f''(x)^2\\
f^{(5)}(x) & 5f^{(4)}(x) & 20f''(x)f'''(x)\\
\end{vmatrix}\begin{matrix} \vphantom{x}\\ \vphantom{y} \\\vphantom{z}.\end{matrix}
\end{align*}
\begin{prop}[Huxley]
Assume that $f\in C^5[N,2N]$ such that there exist real numbers $C,T\geq 1$ such that
\begin{align*}
&|f^{(j)}(x)|\leq C^{j+1}j!\times \frac{T}{N^j} \quad (j=1,\cdots,5)\\
&|f^{(j)}(x)|\geq \frac{j!}{C^{j+1}}\times\frac{T}{N^j} \quad(j=2,3)\\
&|D_1(f;x)|\geq 144C^{-8}\times \frac{T^2}{N^6},\quad |D_2(f;x)|\geq 4320C^{-12}\times \frac{T^3}{N^9}.
\end{align*}
Let $0<\delta<\frac{1}{4}$ be a real number. Then we have
\begin{align*}
\mathcal{R}(f,N,\delta)\ll (NT)^{4/15}+N(\delta^{11} T^9)^{1/75}.
\end{align*}
The implied constant depends only on $C$.
\end{prop}
The main term of this result is very good, since in the second derivative test, the main term only yields $(NT)^{1/3}$ and Theorem \ref{thm11} can only give $(NT)^{3/10}$. A smaller main term will give a better upper bound for the number of integer points when the main term dominates. On the other hand, the secondary term of this result is too large, and thus useless in many applications. But this may be improved subject to some additional non-vanishing conditions of certain quite complicated determinants \citep[p. 403]{huxley}.
\section{Applications}
\subsection{Diophantine Inequality}

One interesting application arises naturally from the methods we have been discussing is to estimate the number of solutions of a given Diophantine inequality. Consider the following Diophantine inequality
\begin{align*}
|\alpha n^2+\beta m^2-z|\leq \delta.
\end{align*}
where $\alpha,\beta,z\in \mathbb{N}$ and $\delta$ small. The goal is to estimate the number of non-trivial solutions as $(n,m)$ range over a closed interval, The inequality can be then reformulated as a problem of finding integer points near the planar curve $\mathcal{C}\subset \mathbb{R}^2$ given by 
\begin{align*}
\alpha n^2+\beta m^2 -z=0.
\end{align*}

Similar ideas have also been discussed by Damaris Schindler \citep{schindler2020diophantine} where Diophantine inequalities for generic ternary diagonal forms are in concern. Fix some degree $k\geq 2$ and let $\alpha_2,\alpha_3\in \mathbb{R}_{>0}$, Let $\theta>0$ and consider the inequality
\begin{align*}
|x_1^k-\alpha_2x_2^k-\alpha_3x_3^k|<\theta. \tag{38}\label{38}
\end{align*}
The goal here is to understand when this inequality has a non-trivial solution if the variables $x_i$ are allowed to range over a box of size $P$. With some techniques rather complicated, we would eventually be able to reformulate the inequality (\ref{38}) as a problem of finding rational points near the planar curve $\mathcal{C}\in \mathbb{R}^2$ given by $1-\alpha_2y_2^k-\alpha_3y_3^k=0$. Finding solutions to (\ref{38}) would be translated into studying a counting function of the type with $\delta=\frac{\theta}{P^{k-1}}$
\begin{align*}
N_{\mathcal{C}}(P,\delta)=\sharp \Big\lbrace \frac{p}{q}\in \mathbb{Q}^2 : 1\leq q\leq P, \operatorname{dist} \Big(\frac{p}{q},\mathcal{C}\Big)\ll \frac{\delta}{q}\Big\rbrace.
\end{align*}
\subsection{The Squarefree Number Problem}
Huxley and Sargos's Theorem also can be applied on the squarefree number problem. Without using the theorem, we could still get the basic result as follows \citep[Lemma 5.3]{huxley}.

\begin{lem}
Let $x,y$ satisfy $16\leq y<\frac{1}{4}\sqrt{x}$ and $2\sqrt{y}\leq A<B\leq 2\sqrt{x}$. Then
\begin{align*}
\sum_{x<n\leq x+y}\mu_2(n)=\frac{y}{\zeta(2)}+O\Big(  (R_1+R_2)\log x +A\Big),
\end{align*}
where $R_1=R_1(A,B)$ and $R_2=R_2(B)$ are defined by
\begin{align*}
R_1=\underset{A < N \leq B}{\max} \mathcal{R}\Big( \frac{x}{n^2},N,\frac{y}{N^2} \Big) \quad \text{and} \quad R_2=\underset{N\leq 2x/B^2}{\max} \mathcal{R}\Big( \sqrt{\frac{x}{n}},N,\frac{y}{\sqrt{Nx}} \Big).
\end{align*}
\end{lem} 
Now suppose first that $y\leq x^{4/9}$, if we use the above lemma with $A=2x^{2/9}$ and $B=x^{1/3}$, by the restriction stated in the lemma, we have $A>2y^{1/2}$ and hence
\begin{align*}
\sum_{x<n\leq x+y}\mu_2(n)=\frac{y}{\zeta(2)}+O\Big(  R_1(2x^{2/9},x^{1/3})\log x+R_2(x^{1/3})\log x+x^{2/9}   \Big).
\end{align*}
Notice that. The restriction on the second derivative test i.e. Theorem \ref{thm3} is satisfied by the function in its respective range of summation. So using Theorem \ref{thm3} we can get
\begin{align*}
R_2\Big( x^{1/3}  \Big)\ll \underset{N\leq 2x^{1/3}}{\max} \Big( (Nx)^{1/6}+y^{1/2}(N^3x^{-1})^{1/4} \Big)\ll x^{2/9}+yx^{-2/9}.
\end{align*}
Furthermore, if we use the Theorem of Huxley and Sargos with $k=3$ for $R_1(2x^{2/9},x^{1/3})$, we will have
\begin{align*}
R_1(2x^{2/9,x^{1/3}}) &\ll \underset{2x^{2/9}<N\leq x^{1/3}}{\max}\Big\lbrace   (Nx)^{1/6}+(Ny)^{1/3}+N(yx^{-1})^{1/3} \Big\rbrace\\
& \ll x^{2/9}+x^{1/9}y^{1/3}.
\end{align*}
Thus, if $y\leq x^{1/3}$, we will have
\begin{align*}
\sum_{x<n\leq x+y}\mu_2(n)=\frac{y}{\zeta(2)}+O\Big(x^{2/9}\log x\Big).
\end{align*}
Which means that there exists a constant $c_0>0$ such that if $c_0x^{2/9}\log x\leq y<\frac{1}{4}\sqrt{x}$, the interval $[x,x+y]$ contains a squarefree number is since by the condition we already have $y\geq 16$, to guarantee that the RHS non-zero, we only need to make sure $\frac{y}{\zeta (2)}$ is bigger than the approximated term $O(x^{2/9}\log x)$ for some such $c_0$.
\section*{Acknowledgements}
This manuscript is derived from my third-year bachelor's dissertation at the University of Warwick. I would like to express my sincere gratitude to Dr. Simon L. Rydin Myerson for his invaluable advice and supervision throughout the development of this essay. Additionally, I acknowledge the exceptional work of Bordellès; without his beautifully translated and meticulously organized textbook, Arithmetic Tales, this essay would not have been possible.

\newpage
\nocite{*}
\bibliography{Biblio}

\end{document}